\documentclass[11pt]{amsart}
\usepackage{amsmath,amsthm,amscd,amssymb, color}
\usepackage{latexsym}
\usepackage[colorlinks, citecolor = blue,backref]{hyperref}
\usepackage[alphabetic]{amsrefs}
\usepackage[capitalize, nameinlink]{cleveref}
\usepackage{enumerate}
\usepackage[margin=1.25in]{geometry}
\usepackage{bbm}

\newcommand{\R}{\mathbb R}

\newcommand{\Q}{\mathbb Q}
\newcommand{\C}{\mathbb C}
\newcommand{\Z}{\mathbb Z}
\newcommand{\F}{\mathbb F}

\renewcommand{\phi}{\varphi}
\newcommand{\eps}{\varepsilon}
\newcommand{\frakp}{\mathfrak p}
\newcommand{\frakP}{\mathfrak P}

\newcommand{\fraka}{\mathfrak a}
\newcommand{\frakb}{\mathfrak b}
\newcommand{\frakc}{\mathfrak c}

\newcommand{\frakm}{\mathfrak m}
\newcommand{\frakN}{\mathfrak N}
\newcommand{\frakM}{\mathfrak M}
\newcommand{\frakD}{\mathfrak D}
\newcommand{\frako}{\mathfrak o}
\newcommand{\calO}{\mathcal O}

\newcommand{\calI}{\mathcal I}

\newcommand{\A}{\mathbb A}
\newcommand{\bs}{\backslash}
\newcommand{\bmx}{\left( \begin{matrix}}
\newcommand{\emx}{\end{matrix} \right)}
\newcommand{\Cl}{\mathrm{Cl}}
\newcommand{\Eis}{\mathrm{Eis}}

\newcommand{\St}{\mathrm{St}}

\newcommand{\new}{\mathrm{new}}
\newcommand{\fin}{\mathrm{fin}}

\renewcommand{\mod}{\, \, \mathrm{mod} \, \,}

\newcommand{\one}{\mathbbm{1}}
\newcommand{\two}{\mathbf{2}}

\DeclareMathOperator{\GL}{GL}
\DeclareMathOperator{\PGL}{PGL}
 
\DeclareMathOperator{\Hom}{Hom} 
\DeclareMathOperator{\Emb}{Emb} 
\DeclareMathOperator{\Pic}{Pic} 
\DeclareMathOperator{\ord}{ord} 
\DeclareMathOperator{\vol}{vol} 
\DeclareMathOperator{\Ram}{Ram} 
\DeclareMathOperator{\Ad}{Ad}

\newtheorem{lemma}{Lemma}
\numberwithin{lemma}{section}
\newtheorem{prop}[lemma]{Proposition}
\newtheorem{thm}[lemma]{Theorem}
\newtheorem{cor}[lemma]{Corollary}

\theoremstyle{definition}
\newtheorem{ex}[lemma]{Example}

\theoremstyle{remark}
\newtheorem{rem}[lemma]{Remark}

\numberwithin{equation}{section}

\begin{document}

\title{Exact double averages of twisted $L$-values}
\author{Kimball Martin}
\address{Department of Mathematics, University of Oklahoma, Norman, OK 73019}
\email{kimball.martin@ou.edu}

\date{\today}

\begin{abstract}
Consider central $L$-values of even weight elliptic or Hilbert modular
forms $f$ twisted by ideal class characters $\chi$ of an imaginary quadratic 
extension $K$.  Fixing $\chi$, and assuming $K$ is inert at each prime
dividing the level, one knows simple exact formulas for averages 
over newforms $f$ of squarefree levels satisfying a parity condition on the
number of prime factors.  These averages stabilize when the 
level is large with respect to $K$  (the ``stable range''). 

In weight 2, we obtain exact formulas for a simultaneous average over 
both $f$ and $\chi$.  We allow for non-squarefree levels with any 
number of prime factors, and ramification
or splitting of $K$ above the level.  Under elementary conditions on the level, 
these double averages are ``stable'' in all ranges.
Two consequences are generalizations of the aforementioned stable (single) averages 
and effective results on nonvanishing of central $L$-values.
\end{abstract}

\maketitle

%
%

\section{Introduction}
 
\subsection{Prime power level}
We first explain our results in the case of
prime power levels over $\Q$.
Denote by $S_2(N)$ the space of weight 2
elliptic modular forms of level $\Gamma_0(N)$.
Let $\mathcal F^\new(N)$ denote the set of level $N$ newforms in $S_2(N)$.
Denote by $L(s, -)$ a completed $L$-function with center
$s= \frac 12$, and by $L_\fin(s,-)$ the nonarchimedean part of $L(s,-)$.

Suppose $N$ is a prime, $K/\Q$ an imaginary quadratic field of discriminant $D_K$ 
which is inert at $N$, and
$\chi$ a character of $\Cl(K)$.  Assuming $D_K$ is odd, \cite{MR} proved an exact formula for weighted 
averages of the central values $L(\frac 12, f,  \chi)$,
where $f$ ranges over $\mathcal F^\new(N)$.  In the \emph{stable range} $N > |D_K|$,
these averages simplify to yield:
\begin{equation} \label{eq:MR}
  \frac{u_K^2 \sqrt{-D_K}}{8\pi^2} \sum_{f \in \mathcal F^\new(N)} \frac{L_\fin(\frac 12, f, \chi)}{(f,f)} = h_K(u_K - \delta_{\chi,1}\frac{12 h_K}{N-1}).
\end{equation}
Here $\delta_{\chi,1}$ is the Kronecker delta and $u_K = \frac{\# \frako_K^\times}2$.
 The work \cite{MR} also allows for weighting by Fourier coefficients of $f$ and arbitrary even weights.  From the generalization in \cite{FW}, \eqref{eq:MR} still holds when $D_K$
 is even.

Our main theme is that the arithmetic of quaternion algebras allows
us to provide exact formulas for any range of $(N, |D_K|)$ when we
average over both $f$ and $\chi$.  
These double average value formulas are strikingly simple for ``balanced'' levels, and yield
stable single average value formulas generalizing \eqref{eq:MR}.

\begin{thm} \label{thm:prime}
Suppose $N = p^{2r+1} \ge 11$ for a prime $p$ with $N \ne 27$,
and $K/\Q$ is an imaginary quadratic field which is not split at $p$. 
If $r > 0$, assume $K/\Q$ is inert at $p$.
Put 
\[ C_K = \frac 12 {e_p(K/\Q) u_K^2} \sqrt{-D_K}, \] 
where $e_p(K/\Q) = 1$ if $K/\Q$ is inert at $p$ and $e_p(K/\Q)=2$ if $K/\Q$ is ramified at $p$. 

\begin{enumerate}[(i)]
\item
If either $p \equiv 1 \mod 12$ or $u_K > 1$, we have the exact double average
formula
\begin{equation} \label{prime:double}
 \frac{C_K}{4\pi^2} \sum_{\chi \in \widehat{\Cl}(K)} \sum_{f \in \mathcal F^\new(N)} \frac{L_\fin(\frac 12, f, \chi)}{(f,f)} =
\begin{cases}
  h_K^2 ( u_K-\frac{12}{p-1}) & r = 0, \\
  h_K^2 u_K(1 - \frac 1{p^2}) & r \ge 1.
\end{cases}
\end{equation}

\item 
If $p > |D_K|$, we have the stable (single) average formula
\begin{equation} \label{prime:single}
 \frac{C_K}{4\pi^2} \sum_{f \in \mathcal F^\new(N)} \frac{L_\fin(\frac 12, f, \chi)}{(f,f)} = 
 \begin{cases}
h_K (u_K - \delta_{\chi,1} \frac{12 h_K}{p-1} ) & r=0, \\
h_K u_K (1 - \frac 1{p^2}) & r \ge 1.
\end{cases}
\end{equation}
\end{enumerate}
\end{thm}

Since the formula in (i) only depends on $h_K$ and $u_K$, by analogy with
\cite{MR}, we think of this double average as being ``stable'' in all ranges when 
$p \equiv 1 \mod 12$ (in which case we call $N=p^{2r+1}$ balanced) or if $u_K > 1$.  
If we do not assume that $p \equiv 1 \mod 12$, then
one still has a double average value formula but it involves 
height pairings and is less elementary.  One can also 
weight the $L$-values by Fourier coefficients of $f$ as in \cite{MR}, but then the 
formulas also involve certain representation numbers of quadratic forms.  
We also treat levels $N=p^2$ when $K$ is ramified at $p$, and allow splitting of
$K$ above the level when $N$ has multiple prime factors.  The issue with
levels $N=p^{2r}$ with $r > 1$ is explained below.

The most interesting, or at least classical, case of (ii) is when $\chi=1_K$
is trivial, so that our twisted $L$-function factors as
$L(s,f,1_K) = L(s,f)L(s,f \otimes \eta_K)$.
In this case if $N=p^{2r+1}$ and $K$ is split at $p$, then we are forced to have 
$L(\frac 12, f, 1_K) = 0$ because the root number is $-1$.  So for prime power 
level $N$ and imaginary quadratic $K$, the cases of most
interest are when $K$ is inert or ramified at $p$, as in the theorem.

Using an extension of (i) to allow $p \not \equiv 1 \mod 12$, we still have elementary
upper and lower bounds on the double average.  We leverage this 
to get lower (and upper) bounds for the single average \eqref{prime:single}
when $\chi = 1_K$ is trivial.
These bounds are valid outside of the stable range $p > |D_K|$ in (ii),
and yield the following sharper \emph{effective} non-vanishing result.

\begin{cor} \label{cor:12}
 Fix an imaginary quadratic field $K$.  Then for any odd power of a prime
$N = p^{2r+1}$ such that $p$ is inert in $K$, there exists a newform $f \in S_2(N)$
such that $L(\frac 12, f) L(\frac 12, f \otimes \eta_K) \ne 0$,
provided that  (i) $p > \frac{12 h_K}{u_K} + 1$ if $N=p$ is prime; or (ii) $p^2 > \frac{3h_K}{u_K}$ if $N > 27$ is not prime.
\end{cor}

The prime level case is already in \cite[Theorem 3]{MR}.  See \cref{lowbdN} for
an analogue in level $p^2$.
We remark that effective non-vanishing results of this type in levels $p^2$ and $2p^2$ 
have been applied to generalized
Fermat equations in \cite{ellenberg} and \cite{BEN}.

\medskip
The prime level case of \cref{thm:prime}(i) (and its extension to general $p$)
follows easily from Gross's $L$-value formula \cite{gross}.  (Gross's formula was also
the basis of the proof of the 
weight 2 case of \cite{MR}.) However, to our knowledge, this had not been observed
before.  The main results of \cite{MR} have been extended to squarefree levels $N$
which have an odd number of prime factors in \cite{FW} (which also treats 
totally real base fields).  
Note that \cite{FW} also requires $K$ to be inert at each prime dividing
the level.  
The approach there is via a relative trace formula, and is based on the 
central $L$-value formula in terms of periods on quaternion algebras from \cite{MW}.

Our approach also uses the relative trace formula from \cite{MW}, but
instead of computing geometric orbital integrals as in \cite{FW}, we directly
compute averages of periods using the arithmetic of quaternion algebras.
The main technical difficulties in this work arise from treating levels which
are both not squarefree and which may have an even number of prime factors.
A novel feature of our approach is that it allows $K$ to be split or ramified at some primes dividing the level when $\chi=1_K$. 
Moreover, our approach to compute period averages is much simpler than
the geometric trace formula calculations in \cite{FW}.  On the other hand,
we do not get an independent proof of the stable range formulas,
nor is it clear how well our method could treat higher weights.
(However, the case of weight 2 where our
method readily applies is the precisely the case in which
 the geometric trace formula calculations are the most complicated.)  
That said, we expect that blending our methods with those of \cite{FW} 
should allow one to extend \cite{FW} to more general levels and ramification
behavior of $K$ in arbitrary even weight.  
(E.g., see \cref{rem:extend-FW}.)

\subsection{A more general class of levels} \label{sec:12}
In this paper, we treat a broad class of level structures for parallel weight 2 Hilbert
modular forms over a totally real base field of class number 1.  However, 
we will restrict to elliptic modular forms in the introduction, as well as for some of
our more explicit results at the end.  

Roughly, we can treat levels $N$ such that if $p^{2r}$ sharply divides $N$ then
$r \le 1$.  However, there is also a parity condition
if there is no prime sharply dividing $N$.  The quadratic fields $K$
can have arbitrary ramification at primes sharply dividing $N$, 
but must be ramified at $p$ when $p^2$ sharply divides $N$ and 
inert at $p$ when $p^3 | N$.  Here is the precise setup.

By a nice level type $(N_1, N_2, M)$ we mean pairwise coprime positive integers
$N_1, N_2$ and $M$ such that (i) each prime dividing $N_1$ occurs to an odd power;
(ii) $N_2$ is the square of a squarefree number; (iii) $ \# \{ p | N_1 N_2 \}$ is odd; and (iv)
$M$ is squarefree.  We say an imaginary quadratic field $K$ is 
$(N_1, N_2, M)$-admissible if $K/\Q$ is non-split at each $p | N_1 N_2$, split at
each $p | M$, ramified at each $p | N_2$, and unramified at each $p$ such that
$p^3 | N_1$.  

From now on, assume that $N=N_1 N_2 M$ with $(N_1, N_2, M)$ a nice level type,
and that $K$ is an $(N_1,N_2, M)$-admissible imaginary quadratic field.
Our main results for elliptic modular forms extend and refine
\cref{thm:prime} such $N$ and $K$.   For general nice level types, our formulas
will involve a contribution from smaller levels, though when $M=1$ we will be
able to isolate a sum over newforms.  We address the reasons for our
restrictions in the outline of the proof below.

Let $\mathcal F(N)$ be the set of normalized eigenforms $f \in S_2(N)$ 
which are newforms of some level $N' | N$ such that
(i) $\ord_p(N')$ is odd for each $p | N_1$, and 
(ii) the local representation of $\GL_2(\Q_p)$ associated to $f$ is discrete series 
for each $p | N_2$.  For a character $\chi$ of $\Cl(K)$, denote by
$\mathcal F(N; \chi)$ the subset of $f \in \mathcal F(N)$ such that
$a_p(f) = \chi_p(\sqrt{-D_K})$ for each $p | N$ which ramifies in $K$.

Put
\begin{equation}
 C(K; N) =  2^{\# \{ p | \gcd(D_K, N) \} -1} 
u_K^2 \sqrt{-D_K} \prod_{p | N_2} \frac 1{1+p^{-1}}.
\end{equation}
Note that since $\{ p | N_2 \} = \{ p | N : \ord_p(N) = 2 \}$, we see $C(K; N)$ only 
depends on $N$ and not the precise choice of triple $(N_1, N_2, M)$, which in general
is not uniquely determined by $N$.  Set
\[ c(N_1, N_2, M) = \frac{12}N \prod_{p | N_1 N_2} \frac 1{1-p^{-1}} \prod_{p | N_2 M}
\frac 1{1+p^{-1}}. \]

We also define a weight $\Lambda_N(f,\chi) = \prod_{p | N} \Lambda_p(f,\chi)$.
Here the local factor $\Lambda_p(f,\chi)=1$  if $p | N_1$,
is defined by \eqref{lambdaN2} when 
$p | N_2$, and by \eqref{lambdaM} or equivalently \eqref{lambdaMQ} when $p | M$.   When $p | N_2$, the
factor $\Lambda_p(f,\chi)$ does not actually depend on $\chi$, only whether
$f$ is $p$-new and whether $f$ is $p$-minimal.

Let $N_1' = \prod p^{\ord_p(N_1)}$ where $p$ runs over primes such that
$\ord_p(N_1) > 1$.
Let $\omega'(n)$ be the number of odd prime divisors of $n$.
Denote by $\deg T_m$ the degree of the Hecke operator on $T_m$,
which is 1 if $m=1$ and $p+1$ if $m=p$ is a prime not dividing $N$.  (More
precisely, it is the degree of $T_m$ on a certain space of quaternionic
modular forms---see \cref{sec:41} for our precise definition.)  

For a newform $f$, let $N_f$ denote the (exact) level of $f$, and
 $(f,f) = (f,f)_{N_f}$ be the usual Petersson norm of $f$ with respect to the
standard measure on $X_0(N_f)$. 

Our main result, \cref{main-thm}, specialized to the case $F=\Q$ says the following,
with remaining notation defined below.

\begin{thm} \label{main-thm:Q}
Let $m \ge 1$ be coprime to $N_1' N_2$.  Then
\begin{multline*} 
\frac{C(K; N)}{4 \pi^2} \sum_{\chi \in \widehat{\Cl}(K)}   \sum_{f \in \mathcal F(N; \chi)} a_m(f) \Lambda_N(f,\chi) \frac{N_f}N  \frac{ L_\fin(\frac 12, f,  \chi)} {(f,f)} = \\
h_K \left(\sum w_i h_i a_{ii}(m) - \delta^+(N, m) \deg T_m 2^{\omega'(N_2)}h_K
c(N_1, N_2, M) \right).
\end{multline*}
In particular, if $m=1$, we have
\begin{multline*} 
\frac{C(K; N)}{4 \pi^2} \sum_{\chi \in \widehat{\Cl}(K)}   \sum_{f \in \mathcal F(N; \chi)} \Lambda_N(f, \chi) \frac{N_f}N  \frac{L_\fin(\frac 12, f,  \chi)} {(f,f)} = \\
h_K \left( \sum w_i h_i  - 2^{\omega'(N_2)}h_K
c(N_1, N_2, M)\right). 
\end{multline*}
\end{thm}

The weighting of our average $L$-values by $\Lambda_N(f,\chi) \frac {N_f}N$
is simply for two reasons:  to account for oldforms at $p | M$, and because
we wrote the above average in terms of the Petersson norm $(f,f)$ rather than 
$L(1, f, \Ad)$.  The factors $\Lambda_p(f,\chi)$ are just 1 for $p | M$ when $f$
is $p$-new, and the factors $\frac {N_f}{N}$ and $\Lambda_p(f,\chi)$ for $p | N_2$
are not needed if one replaces $(f,f)$ with $L(1,f, \Ad)$ (see \cref{main-thm}.)

We will now explain the remaining quantities in this theorem, but the
main point is that the quantity $\sum w_i h_i$ is approximately 
$h_K$ (or exactly $u_K h_K$ if $u_K > 1$), and is exactly $h_K$  under either elementary conditions on $(N_1, N_2, M)$
or when we are in a stable range.  

Let $B/\Q$ be the definite quaternion algebra with discriminant 
$D_B = \prod_{p | N_1 N_2} p$.  Then the condition that $K$ is 
$(N_1,N_2,M)$-admissible implies that there is a special order (in the sense of
Hijikata--Pizer--Shemanske \cite{HPS}) $\calO \subset B$ of level type $(N_1, N_2, M)$
in which $\frako_K$ embeds.  In particular, if $N$ is squarefree,
we simply mean that $\calO$ is an Eichler order of level $N$ in $B$.  
See \cref{sec:spec} for the general definition.  We remark that $c(N_1, N_2, M)$
is simply the reciprocal of the mass of $\calO$.

Let $\Cl(\calO) = \{ \calI_1, \dots, \calI_n \}$ denote the set of (invertible)
right $\calO$-ideal
classes of $B$.  Let $w_i$ be one half of the number of units in the left order of $\calI_i$.
We usually have $w_i = 1$, and if $N \ge 5$ then each $w_i \le 3$ (see \cref{lem:wi}).
Fix an embedding $\frako_K \subset \calO$, which induces an ideal
class map $\Cl(K) \to \Cl(\calO)$.  Let $h_i$ be the size of the preimage of $\calI_i$
under this map.  Hence $\sum h_i = h_K$, and we see that
$h_K \le \sum w_i h_i \le 3 h_K$.  We note that $\sum w_i h_i$
can be interpreted in terms of heights of special cycles as in \cite{gross}.
This explains all of the notation in the $m=1$ case.

For $m \ge 1$, $\delta^+(N, m)$ is either 0 or 1.  It is always 1 if $m=1$ or
if $N_2 = 1$.  See \cref{prop:per-avg} for the full definition, where it is denoted
$\delta^+(\calO,m)$.  Finally, $a_{ii}(m)$ denotes the $i$-th diagonal
element of the $m$-th Brandt matrix associated to the given ordering of $\Cl(\calO)$.
This can be expressed in terms of the number of ways a quadratic form
associated to (the left order of) $\calI_i$ represents $m$---see 
\cref{sec:brandt}.

We say a triple $(N_1, N_2, M)$ is balanced if it satisfies the
conditions of \cref{bal-crit}.  These are simple elementary conditions on the primes
dividing $N_1, N_2$ and $M$, that force  $\sum w_i h_i = h_K$ for any 
$(N_1,N_2,M)$-admissible $K$.  For instance, the triple $(N_1, N_2, M)$
is balanced if $N_2 > 9$, or if there exists a $p \equiv 1 \mod 12$ dividing $N_1$,
or if there exists a $p \equiv 11 \mod 12$ dividing $M$.  Hence for balanced
level types, the above double average formula, at least for $m=1$,  
is completely elementary.

To get from \cref{main-thm:Q} to \cref{thm:prime}(i), we want to isolate the newform
contribution at level $N$.  This is possible, at least for $m=1$, when
$\sum w_i h_i = h_K$ in both level $N$ as well as any relevant
level $N' | N$ for which we need to remove the contribution of oldforms.
(This is the reason for the at-first-glance curious exclusion of $N=27$---cf.\ \cref{ex:27}.)
This idea leads to a generalization of \cref{thm:prime}(i) to where one restricts to forms
which are new away from $M$.  (The factors $\Lambda_N(f, \chi)$ prevent
us from exactly isolating forms which are new also at $M$, but we can at least
approximate the contribution from forms new at $M$---e.g., \cref{lowbdM}.)

Let $\mathcal F_0(N)$ be the set of normalized eigenforms $f \in \mathcal F(N)$ 
which are newforms for some level $N' | N$ such that $\ord_p(N') = \ord_p(N)$
for $p | N_1 N_2$, i.e., $f$ is $N_1N_2$-new.
In particular, if $M=1$ each $f \in \mathcal F_0(N)$ is new of level $N$
(though one does not get all newforms if $N_2 > 1$).
For a character $\chi$ of $\Cl(K)$, put 
$\mathcal F_0(N; \chi) = \mathcal F_0(N) \cap \mathcal F(N; \chi)$.

\begin{thm} \label{thm2} Assume that one of the following holds:
\begin{enumerate}
\item the triple $(D_B,1,M)$ is balanced;

\item we are in the stable range $D_B > |D_K|$ with $\gcd(D_B, D_K) = 1$; or

\item $u_K > 1$ and $N \ge 11$ with $N \ne 27$.
\end{enumerate}
Then
\begin{multline*} \frac{C(K; N)}{4\pi^2} \sum_{\chi \in \widehat{\Cl}(K)}   \sum_{f \in \mathcal F_0(N; \chi)} \Lambda_N(f,\chi) \frac{N_f}N  \frac{L_\fin(\frac 12, f,  \chi)} {(f,f)} =
\\ 
h_K^2 \left( u_K \prod_{p | N_1'} \left( 1- \frac 1{p^2} \right) \cdot 
\prod_{p | N_2} \frac p{p+1}
 -   \delta \frac {12}{N}  \prod_{p | N_1 N_2} \frac 1{1-p^{-1}}
\prod_{p | N_2 M} \frac 1{1+p^{-1}} \right),
\end{multline*}
where $\delta = 1$ if $N_1$ is squarefree and $N_2$ is odd; otherwise $\delta = 0$.
\end{thm}

Now one can ask what these results tell us about single averages over $f$
as in \eqref{eq:MR}.  The above theorems clearly give upper bounds on averages
for fixed $\chi$ which are exact, though certainly suboptimal.  More interesting
are exact formulas and lower bounds.  First we state a stable average value formula
generalizing \cref{thm:prime}(ii).

\begin{cor} \label{cor:stable} Suppose
$N_2 = 1$.  Fix a character $\chi$ of ${\Cl}(K)$.  Then, in the stable range $D_B > |D_K|$ with $\gcd(D_B,D_K) = 1$, we have
\begin{multline*} \frac{C(K; N)}{4\pi^2}  \sum_{f \in \mathcal F_0(N; \chi)} \Lambda_M(f,\chi) \frac{N_f}N  \frac{L_\fin(\frac 12, f,  \chi)} {(f,f)} = \\
h_K \left( u_K \prod_{p | N_1'} \left( 1- \frac 1{p^2} \right) 
- \delta \frac{12 h_K}N \prod_{p | N_1} \frac 1{1-p^{-1}}
\prod_{p | M} \frac 1{1+p^{-1}} \right),
\end{multline*}
where $\delta = 1$ if both $N_1$ is squarefree and $\chi = 1$, and $\delta = 0$
otherwise.
\end{cor}

The novel aspects of this formula are that $N$ is not required to be squarefree,
 $N$ may be divisible by an even
number of primes, and $K$ may be split at primes dividing $N$.
The restriction to $N_2 = 1$ is not due to a real obstruction in the method, 
but simply to make the
second term on the right (which comes from Eisenstein series) easy to describe.

Our final results are effective lower bounds for average $L$-values 
$L_\fin(\frac 12, f, 1_K)$, where $f$ runs over level $N$ newforms in 
$\mathcal F(N; 1_K)$.  For simplicity, we just treat 2 cases:
\begin{enumerate}
\item $N=p^r N_0$, where $N_0$ is a squarefree product of an even number of primes, $r$ is odd or 2, and $K$ is inert or ramified at each prime dividing $N$ (necessarily
ramified at $p$ if $r=2$, and inert at $p$ if $r > 2$; here $M=1$)

\item $N=N_1 p$ is a squarefree product of an even number of primes, and
$K$ is inert or ramified at each prime dividing $N_1$ and $K$ is split at $p$ (here $M=p$)
\end{enumerate}
In each of these situations, the lower bound implies a generalization of \cref{cor:12}:
when $p$ or $p^2$ is larger than an explicit multiple of $h_K$, there exists a newform
$f \in S_2^\new(N)$ such that $L(\frac 12, f) L(\frac 12, f \otimes \eta_K) \ne 0$.
See \cref{sec:lowbds} for precise statements.

As a check on our formulas, we note that our single stable average formulas
match with those in \cite{MR} and \cite{FW} when $N = N_1$ is squarefree and
$K$ is inert at each $p | N$.  We also present a few examples in \cref{sec:ex}
that provide a numerical check on our formulas in other situations.

\subsection{Methods}
The proof of the exact double average formula has two main steps.  First, one computes
averages of squares of periods for trivial weight quaternionic modular (i.e., automorphic)
forms associated to an order $\calO$ in a definite quaternion algebra $B$.  
These quaternionic modular forms are $\C$-valued functions on the finite set 
$\Cl(\calO)$ which have a Hecke action, 
and we can think of a basis of quaternionic Hecke eigenforms 
as being analogous to the complex irreducible characters of a finite group $G$.
The analogue of column orthogonality of characters of $G$ (or a generalization
when one weights by Hecke eigenvalues) leads easily to
an exact double average formula for periods.  This much is carried out
for general orders $\calO$ in definite quaternion algebras over
arbitrary totally real fields in \cref{sec:per}.

Second, one relates periods of quaternionic modular forms to central $L$-values
with the relative trace formula.  For this, one first needs to precisely
relate the quaternionic modular forms for $\calO$ to classical 
modular forms, which is a refinement of the Jacquet--Langlands correspondence.  
We carried this out for special orders $\calO$
of level $N=N_1 N_2 M$ in \cite{me:basis}, under the first 3 defining conditions 
(i)--(iii) for $(N_1, N_2, M)$ to be a nice level type.  Here $M$ need not be
squarefree.  The condition that $N_2$ must be the square of a squarefree
number arises due to complications in the precise description of the
space of quaternionic
oldforms when higher even powers of primes sharply divide $N_2$. 

To finish the second step, we need to precisely relate the quaternionic
periods to $L$-values of modular forms.  This is essentially carried out
in \cite{MW}, but the technical complication is that in general one also needs to
relate periods of quaternionic \emph{oldforms} to $L$-values.  This requires
an understanding of the local quaternionic oldforms at primes $p | N_1 N_2$ 
from \cite{me:basis} and a calculation of local spectral distributions for oldforms
when $p | M$.   The latter is possibly quite complicated in general, and we only
carry this out when $M$ is squarefree and $K$ is split at each $p | M$.
However, in principle, one could extend this to allow 
$M$ to be non-squarefree and $K$ to be split or ramified at $p | M$.

The remaining conditions on the admissibility of $K$ and the restriction
of $f \in \mathcal F(N; \chi)$ are what we need to
see that the relevant periods appear when working with special orders of level $N$ in 
$B$.  In general, provided that an $L$-function $L(s,f,\chi)$ has root number $+1$, 
the central value corresponds to a period on a unique quaternion algebra, 
which may or may not be $B$.

This leads to \cref{main-thm:Q}.  Then one derives \cref{thm2} by isolating the
$N_1N_2$-new contribution using the inclusion-exclusion principle.

The restriction to base fields of class number 1 in \cref{main-thm} for Hilbert
modular forms is simply due to the fact that
the relative trace formula of Jacquet used in the $L$-value formula of \cite{MW} 
was only established in for representations whose base change to $K$ has trivial 
central character.

\begin{rem} \label{rem:extend-FW}
Our calculation of local spectral distributions for $p | M$ is based
off of a similar calculation from \cite{FMP} at a prime $p$ sharply dividing $M$ such
that $K$ is inert at $p$ and $\chi_p$ is \emph{ramified}.  If $K$ is inert at $p | M$
and $\chi_p$ is unramified (as in the present paper), one does not see the relevant
$L$-values with periods on $B$, but on a quaternion algebra which is ramified at $p$. 
It should be possible to use our calculations in \cref{sec:loc-spec} to
extend the average value formulas in \cite{FW} and \cite{FMP} to levels which
are products of even numbers of primes, and allow for $K$ to be split at
primes dividing the level, with the caveat that now one needs to weight the $L$-values
of oldforms by the factors $\Lambda_N(f,\chi)$.
\end{rem}

Now we briefly explain how to derive the lower bounds and stable formulas for single
averages.  Fixing $\chi$, the average of $L$-values $L(\frac 12, f, \chi)$ for
$f \in \mathcal F(N; \chi)$ plus an Eisenstein contribution is essentially the height 
$\langle c_{K,\chi}, c_{K,\chi} \rangle$
of a special divisor $c_{K,\chi}$.  
These heights are maximized when $\chi=1_K$ is the trivial
character, hence the single average of $L(\frac 12, f, 1_K)$ must be at least
$\frac 1{h_K}$ times the double average over $f$ and $\chi$, minus the Eisenstein
contribution.  This leads to the lower bounds in \cref{sec:lowbds}, and then
\cref{cor:12}.

We now \emph{define} $(N,K)$ to be in stable range if $\langle c_{K,\chi}, c_{K,\chi} \rangle$
is dependent only on $K$ (and our triple $(N_1,N_2,M)$) but not $\chi$.
Then, in the stable range,  the average over $f \in \mathcal F(N; \chi)$ for any fixed $\chi$
is given by $\frac 1{h_K}$ times the double average over $f$ and $\chi$, 
minus the Eisenstein contribution.  The converse is also true, and so by comparing
our double average formulas with the stable averages in \cite{MR} and \cite{FW},
we can conclude that their stable range is contained in our definition of stable range.  
In other words, we use \cite{MR} and \cite{FW} to verify that $D_B > |D_K|$ implies
$(D_B,K)$ is in the stable range by our definition.  Then we use this to deduce 
$(N,K)$ is also in the stable range.  This leads to \cref{thm:prime}(ii) and \cref{cor:stable}.

\begin{rem} Since we make a variety of different hypotheses on
our level types $(N_1, N_2, M)$ and the quadratic extension $K$ for different
results, and we also impose running assumptions at certain points, 
in \cref{sec:appendix}
we present a table summarizing which hypotheses are used in which results for the convenience of the reader.
\end{rem}

\subsection{Related work} \label{sec:rel}
Here we briefly discuss some additional related work.

First, we know of two works on stable averages which allow non-squarefree level.
In \cite{nelson}, Nelson established a quite general formula for
Rankin--Selberg averages $L(\frac 12, f \times g)$ over $\Q$.  Here $g$ is a fixed
form and one averages over an orthogonal basis of cusp forms $f$ of some weight 
$k \ge 4$.  The full formula is quite complicated, but it stabilizes in various situations,
including large prime-power levels.  However, Nelson does not get averages over
newforms (or in weight 2). 

Recently, Pi \cite{pi} proved a certain stable average formula similar to that of \cite{FW},
which allows for squares and cubes dividing the level.  More precisely,
Pi averages over representations which are \emph{fixed} depth 0 or simple supercuspidals $\pi_v$ at a given set of places.  This restriction on the type
of local representations forces the bound for the stable range to be quite large.  
Moreover, the character $\chi$ is prescribed according to the local representation 
type of $\pi_v$, so one cannot always take $\chi = 1_K$.  Consequently,
even for level $p^3$ it does not give a stable average result as in \cref{thm:prime}(ii).

We remark that there are also a number of asymptotic results on averages
for a fixed $\chi$.  See for instance \cite{ST} for weights $\ge 6$
when $\chi = 1_K$. 

The main consequences of our double average value formula that we considered
here are about single averages over $f$ for fixed $K$ and $\chi$.
In particular, we thought of $K$ as being fixed and $N$ varying.  However
one nice feature of our double average formula, especially the
stable double average in  \cref{thm2}, is that we may simultaneously
vary $N$ and $K$.  

For a fixed form $f$ with $|D_K| \to \infty$, an asymptotic for averages over
$\chi \in \widehat{\Cl}(K)$ was established in \cite{MV}, which led to
a quantitative lower bound on the number of nonvanishing twists $L(\frac 12, f, \chi)$.  
The authors assumed
$f$ is weight 2 of prime level over $\Q$ for simplicity, though their method 
(using equidistribution of special cycles) clearly generalizes. 
This asymptotic average was extended in \cite{LMY} where now one can vary
both $f$ and $K$.  Our double average value formula at least provides exact
upper bounds on such averages in more general situations, though it is
too crude to obtain averages for a fixed $f$ (unless $\# \mathcal F_0(N) = 1$).

Finally, we remark that many earlier average value results have yielded
subconvexity results.  One might wonder if it is possible to obtain
some form of subconvexity from our double average value formula.
The problem is that our double average involves
too many $L$-values.  Indeed, the upper bound one naively from gets
\cref{main-thm:Q} is that $L(\frac 12, f, \chi) \ll (f,f) |D_K|^{\frac 12 + \eps}$,
which is no better than convexity.

\subsection*{Acknowledgements}
We thank Dan Collins for discussions about computations of Petersson
norms.  We thank the referee for a careful reading and detailed, thoughtful
comments which led to several improvements in the exposition.
This work was supported in part by by a grant from the 
Simons Foundation (512927, KM).

\section{Notation} \label{sec:notation}

Let $F$ be a number field with integer ring $\frako_F$ and discriminant $D_F$.
We denote a place of $F$ by $v$.  If $A$ is an $F$- or $\frako_F$-algebra,
denote by $A_v$ the localization at $v$, and by $\hat A = \prod'_{v < \infty} A_v$
the restricted direct product with respect to $\{ R_v \}_{v < \infty}$ where $R \subset A$
is a maximal $\frako_F$-order. 
In particular, $\hat \frako_F = \prod_{v < \infty} \frako_{F,v}$
and the adele ring is $\A_F = \prod'_v F_v = \hat F \times F_\infty$.  The ideal class
group of $F$ is denoted by $\Cl(F)$, which we identify with both
$F^\times \bs \A_F^\times / \hat \frako_F^\times F_\infty^\times$ 
and $F^\times \bs \hat F^\times / \hat \frako_F^\times$.
Put $h_F = \# \Cl(F)$ and let
 $\widehat{\Cl}(F)$ denote the character group of $\Cl(F)$.

For a finite place $v$ of $F$, we use $\frakp_v$ to denote both the associated
prime ideal of $\frako_F$ and the maximal ideal of $\frako_{F,v}$---when the
distinction matters, the context will make this notation clear.  
Throughout $\frakN$ will be a nonzero ideal in $\frako_F$, and
$\ord_v(\frakN)$ denotes the $\frakp_v$-adic valuation of $\frakN$.
Let $\varpi_{F,v}$ denote a uniformizer in $\frako_{F,v}$ and 
$q_v$ be the size of the residue field of $\frako_{F,v}$.
When we specialize to $F=\Q$, we typically use the roman form of
the corresponding fraktur letters to denote the positive generator of an ideal, e.g.,
$N$ is the positive generator of a non-zero ideal $\frakN \subset \Z$.

In what follows, $F$ is a totally real number field, and we often write
$\frako = \frako_F$, $\frako_v = \frako_{F,v}$, etc.  
In \cref{sec:Lvals}, we further assume $h_F = 1$.
Denote by
$B$ a definite quaternion algebra over $F$, and $\calO$ 
an $\frako$-order in $B$.  
Let $\Cl(\calO)$ be the set of (invertible) 
right $\calO$-ideal classes in $B$.  As with number fields, we identify
$\Cl(\calO) = B^\times \bs B^\times(\A_F) / \hat \calO^\times B^\times(F_\infty)
= B^\times \bs \hat B^\times / \hat \calO^\times$.
Write $n = h(\calO) = \# \Cl(\calO)$ for the class number of $\calO$.

For a simple algebra $A/F$ (we have in mind a quaternion algebra $B/F$
or a quadratic field extension $K/F$), let $N_{A/F}$
 denote the reduced norm map from $A$ to $F$.
When understood or unimportant, we often omit the subscript, and simply
write $N$ for the norm map.  Further, denote by $\Ram(A)$ the set of \emph{finite} primes
$v$ of $F$ at which $A$ ramifies, and ${\mathfrak D}_A = \prod_{v \in \Ram(A)} \frakp_v$ 
the (reduced)  relative discriminant of $A$ (over $F$).

For a quadratic extension $K/F$ of number fields, we denote by
$\eta_K = \eta_{K/F}$ the associated quadratic idele class character of $F$, and analogously
for quadratic extensions of local fields.


Excluding the right regular representation, 
all representations are assumed to be irreducible.
We often use $1$ to denote a trivial character.  However, in the case of a trivial
local or global Hecke character over a field $k$, we sometimes write $1_k$ for clarity.

\section{Quaternionic modular forms}

In this section, we discuss quaternionic modular forms for arbitrary orders
$\calO \subset B$, but for some results we need to restrict to
special orders.

\subsection{Special orders} \label{sec:spec}
Fix a definite quaternion algebra $B/F$.
A class of quaternionic orders $\calO \subset B$ generalizing
Eichler orders was introduced was introduced in \cite{HPS}, which
the authors termed special.  
First we recall the local notion of special orders.  

Let $v$ be a finite place of $F$.  If $B_v$ is split, a local \emph{special}
order $\calO_v \subset B_v \simeq M_2(F_v)$ of level $r$ (or level $\frakp_v^r$) 
is a local Eichler order of level $r$, i.e., an order conjugate to 
$\bmx \frako_v & \frako_v \\ \frakp_v^r & \frako_v \emx$.

Now suppose $B_v$ is division and $E_v/F_v$ is a quadratic
extension of local fields.  Let $\calO_{B,v}$ denote the unique
maximal order of $B_v$ and $\frakP_v$ its unique maximal (2-sided)
ideal.  Set 
\[ \calO_r(E_v) = \frako_{E,v} + \frakP_v^{r-1}, \quad r \ge 1. \]
A local order $\calO_v \subset B_v$ is \emph{special} if 
$\calO_v \simeq \calO_r(E_v)$ for some $r$, $E_v/F_v$, in
which case we say the level of $\calO_v$ is $r$ (or $\frakp_v^r$)
if $r$ is minimal such that $\calO_v \simeq \calO_r(E_v)$.
In particular, $\calO_v$ is special if and only if it containts the integer ring
of some quadratic subfield, and thus 
when $B_v$ is division special orders are the same as what are called
basic orders in \cite{voight:book}, where it is also shown that 
these are the same as Bass orders.

We recall some facts about local special orders.

\begin{lemma}[\cite{HPS}] \label{lem:HPS}
 Suppose $B_v$ is division.  Then:
\begin{enumerate}[(i)]
\item Every special order in $B_v$ is of the form $\calO_{r}(E_v)$
where either $r \ge 1$ is odd and $E_v/F_v$ is unramified, or $r \ge 1$ is
arbitrary and $E_v/F_v$ is ramified.

\item $\calO_1(E_v) = \calO_{B,v}$ is the maximal order for every
$E_v/F_v$.

\item Suppose $E_v/F_v$ and $E_v'/F_v$ are quadratic extensions
and $r, r' \ge 1$.  If $r \ne r'$, then $\calO_{r}(E_v) \simeq \calO_{r'}(E'_v)$
implies $E_v \simeq E'_v$ is unramified and $|r-r'| = 1$ with
$\min \{ r, r' \}$ odd.

\item If $E_v/F_v$ is unramified and $E'_v/F_v$ is ramified, then
$\calO_r(E_v) = \calO_{r'}(E'_v)$ if and only if $r=r'=1$.

\item Suppose $E_v, E'_v$ are two non-isomorphic ramified 
quadratic extensions of $F_v$.  If $r \le 2$, then
$\calO_r(E_v) \simeq \calO_r(E'_v)$; the converse also holds if
$v$ is odd.
\end{enumerate}
\end{lemma}

See \cite[Theorem 3.10]{HPS} for the converse of (v) when
$v$ is dyadic.

Following \cite{HPS}, we say a global order $\calO \subset B$ is \emph{special of level}
$\frakN$ if $\calO_v$ is special of level $\ord_v(\frakN)$
for each $v < \infty$.  However, $\frakN$ alone does not determine
the local isomorphism type
of $\calO$.  We will always place the following additional
assumption on a global special order $\calO$ of level $\frakN$:
 for $v < \infty$ such that $B_v$ is
division and $r_v := \ord_v(\frakN)$ is odd, $\calO_v \simeq \calO_{r_v}(E_v)$ where
$E_v/F_v$ is the unramified quadratic extension.
Further, when we write $\frakN = \frakN_1 \frakN_2 \frakM$, it will always mean
that $\frakN_1, \frakN_2$ and $\frakM$ are pairwise coprime ideals such that (i)
for finite $v | \frakN_1 \frakN_2$ if and only if $B_v$ is division, and (ii) for
$v | \frakN_1 \frakN_2$, $v | \frakN_1$ if and only if $\ord_v(\frakN_1)$ is odd.
Thus $\frakN_1$ (resp.\ $\frakN_2$) is divisible exactly by the finite primes $v$
ramifying in $B$ such that $\calO_v \simeq \calO_{r_v}(E_v)$ for some $r_v$
where $E_v$ is an unramified (resp.\ ramified)
quadratic extension of $F_v$, and $\frakM$ is divisible exactly by the finite
primes $v$ splitting $B$ where $\calO_v$ is a non-maximal local Eichler order.

Sometimes to emphasize the above conventions, we will say $\calO$ is a special
order of level type $(\frakN_1, \frakN_2, \frakM)$.
Note that by \cref{lem:HPS}, if
$\ord_v(\frakN_2) = 2$ for all $v | \frakN_2$, then any two special orders of level type $(\frakN_1,
\frakN_2, \frakM)$ are locally isomorphic, i.e., in the same genus.

If $\calO$ is special of level type $(\frakN_1, \frako, \frakM)$, we say $\calO$ is of unramified quadratic type.
This means that $\calO_v$ contains the maximal order of an unramified or
split quadratic extension of $F_v$ for each $v < \infty$.
Such orders are nice to work with because then one has
that $N(\hat \calO^\times) = \hat \frako^\times$.
Note that an order $\calO$ is Eichler if and only if $\calO$ is of unramified
quadratic type with $\frakN_1$ squarefree, i.e., $\frakN_1 = {\mathfrak D}_B$
and $\frakN_2 = \frako$.

\subsection{Automorphic forms}
Fix an arbitrary $\frako_F$-order $\calO \subset B$.  We define
the space of quaternionic
modular forms (or automorphic forms) of level $\calO$  on
$B$ with trivial weight to be $M(\calO) = \{ \phi : \Cl(\calO) \to \C \}$.
Set $n = h(\calO) = \# \Cl(\calO)$, and write $\Cl(\calO) = \{ x_1, \dots, x_n \}$.
By abuse of notation, we use also use $x_i$ to denote an element of
$\hat B^\times$ representing the corresponding element of 
$\Cl(\calO) = B^\times \bs \hat B^\times / \hat \calO^\times$.  
Let $\calI_i$ denote a classical right $\calO$-ideal
representing $x_i$.
Let $\calO_\ell(\calI_i) = \{ \alpha \in B : \alpha \calI_i \subset \calI_i \} =
 x_i \hat \calO x_i^{-1} \cap B$ be the left order of $\calI_i$.  
Put $w_i = [\calO_\ell(\calI_i)^\times : \frako^\times]$, which is always
finite.  We define an inner product on $M(\calO)$ by
\begin{equation}
 (\phi, \phi') = \sum_{i=1}^n \frac 1{w_i} \phi(x_i) \overline{\phi'(x_i)}.
\end{equation}
This inner product can also be defined by integrating $\phi$ against
$\bar \phi'$ over $\hat F^\times B^\times \bs \hat B^\times$ with an
appropriate normalization of measure (e.g., see \cite{me:cong}).

Let $\omega: F^\times \bs \A_F^\times \to \C^\times$ be an 
idele class character of $F$ which is trivial at each infinite place.
Let $L^2(B^\times \bs \hat B^\times, \omega)$ be the space of
$L^2$-functions $\phi : B^\times \bs \hat B^\times \to \C$ such that $\phi(zx) = \omega(z)
\phi(x)$ for all $z \in \hat F^\times$, $x \in \hat B^\times$.  Setting
$M(\calO, \omega) = M(\calO) \cap L^2(B^\times \bs \hat B^\times, \omega)$, we see that
\begin{equation}
M(\calO) = \bigoplus M(\calO, \omega),
\end{equation}
where $\omega$ ranges over idele class characters which are
trivial on $\frako_F^\times \times F_\infty^\times = \A_F^\times \cap ( \hat \calO^\times \times B_\infty^\times)$,
i.e., $\omega$ runs over the group $\widehat{\Cl}(F)$ of ideal class characters of $F$.

By an automorphic representation $\pi$ of $B^\times$ with trivial weight
and (central) character $\omega$,
we mean an irreducible unitary subrepresentation of the right regular representation of 
$\hat B^\times$ on $L^2(B^\times \bs \hat B^\times, \omega)$.  
Call $\pi$ cuspidal if $\pi$ is not 1-dimensional.  
We say that $\pi$ occurs in $M(\calO)$ if $\pi \cap M(\calO) \ne 0$, 
which is equivalent to $\pi^{\hat \calO^\times} \ne 0$, and speak similarly
for $M(\calO,\omega)$.  
Accordingly we get decompositions
\begin{equation}
 M(\calO) = \bigoplus \pi^{\hat \calO^\times}, \quad
M(\calO, \omega) = \bigoplus \pi^{\hat \calO^\times},
\end{equation}
where $\pi$ respectively runs over automorphic representations occurring in
$M(\calO)$ and $M(\calO, \omega)$.

We define the Eisenstein subspace of $M(\calO)$ to be
$\Eis(\calO) = \bigoplus \pi^{\hat \calO^\times}$, where
$\pi$ runs over 1-dimensional automorphic representations occurring in
$M(\calO)$.  
All such $\pi$ must be of the form $\mu \circ N_{B/F}$
for some character $\mu : F^\times \bs \A_F^\times \to \C^\times$.   
We may view $N = N_{B/F}$
as a map from $B^\times \bs B^\times(\A_F) / \hat \calO^\times B_\infty^\times$ to
\[ \Cl^+(N(\hat \calO)) := F^\times \bs \A_F^\times / N(\hat \calO^\times) F_\infty^+, \]
where $F_\infty^+$ denotes the totally positive elements of $F_\infty$.
Thus we may take $\mu$ to be a character of $\Cl^+(N(\hat \calO))$.
If $\calO$ is special of unramified quadratic type,
then $\Cl^+(N(\hat \calO))$ is just the narrow class group $\Cl^+(F)$.

Put $\Eis(\calO, \omega) = \Eis(\calO) \cap M(\calO, \omega)$.
Note $\Eis(\calO, 1)$ always contains the constant function $\one$.  In general,
a basis of $\Eis(\calO, \omega)$ is given by the characters $\mu \circ N$, where
$\mu$ runs over the characters of $\Cl^+(N(\hat \calO))$ such that $\mu^2 = \omega$.
Note that for any character $\mu \circ N \in \Eis(\calO)$, its norm 
$(\mu \circ N, \mu \circ N)$ equals the mass of $\calO$,
\[  m(\calO) := \sum_{i=1}^n \frac 1{w_i}. \]

\begin{lemma} \label{mass-form}
Let $\calO$ be a special order of level type
$(\frakN_1, \frakN_2, \frakM)$.  Then
\[
m(\calO) = m(\frakN_1, \frakN_2, \frakM) := 2^{1-[F:\Q]} h_F | \zeta_F(-1) | N(\frakN)
\prod_{v | \frakN_1 \frakN_2} (1-q_v^{-1})
\prod_{v | \frakN_2 \frakM} (1+q_v^{-1}).
\]
\end{lemma}

\begin{proof}
Let $\calO'$ be a special order of level type $(\frakN_1', \frako, \frakM)$
containing $\calO$ where $\frakN_1' = \frakN_1 \prod_{v | \frakN_2} \frakp_v$.
A formula for $m(\calO')$ is given in \cite[(1.6)]{me:cong}.  Since one can
interpret masses of orders as volumes in $\hat B^\times$ with respect
to suitable Haar measures, one has that 
$m(\calO) = m(\calO') \prod_{v | \frakN_2} [(\calO_v')^\times : \calO_v^\times]$.
From \cite[Proposition 2.6]{HPS}, one knows 
$[(\calO_v')^\times : \calO_v^\times] = (q_v+1)q_v^{e_v-2}$ where $e_v = \ord_v(\frakN_2)$.  These two calculations combine to give the lemma.
(The $F=\Q$ case is already stated in \cite[Theorem 6.8]{HPS}.)
\end{proof}

\begin{lemma} \label{EisO1} Let $\calO$ be a special order of level type
$(\frakN_1, \frakN_2, \frakM)$ such that $\ord_v(\frakN_2) = 2$
for each $v | \frakN_2$.  
Then there is a basis of 
$\Eis(\calO,1)$ consisting of $\one$ and the characters $\eta_{E/F} \circ N$, 
where $E/F$ runs over quadratic extensions such that
each finite $v$ ramified in $E/F$ is odd and divides $\frakN_2$.
\end{lemma}

\begin{proof} 
Consider a character $\mu \circ N$ occurring in $M(\calO,1)$
and let $F_v^{(2)}$ denote the subgroup of squares of $F_v^\times$.
Then $\mu^2 = 1$ and $\mu_v$ factors through
$F_v^\times/N(\calO_v^\times) F_v^{(2)}$ for $v < \infty$.  
For $v | \infty$, the only requirement on $\mu_v$ is that it
 factors through $\R^\times / \R_{>0} \simeq \{ \pm 1 \}$.  
For a finite $v \nmid \frakN_2$, we have $N(\calO_v^\times)
= \frako_v^\times$, so 
$F_v^\times/N(\calO_v^\times) F_v^{(2)} \simeq \langle \varpi_v 
\rangle / \langle \varpi_v^2 \rangle$ and we see $\mu_v$ can be
trivial or the unramified quadratic character of $F_v^\times$.

Now suppose $v | \frakN_2$.  To elucidate the assumption on $\frakN_2$,
first just suppose $\ord_v(\frakN_2)=2r$ with $r \ge 1$, and write
$\calO_v = \calO_{2r}(L_v)$ where $L_v/F_v$ is a ramified quadratic
extension.    Then $\calO_v^\times = \frako_{L,v}^\times (1 + \frakP_v^{2r-1})$.
One has $N (1 + \frakP_v^{2r-1}) = 1 + \frakp_v^r$, and thus
$F_v^\times/N(\calO_v^\times) F_v^{(2)} \simeq (\frako_v^\times
/ N(\frako_{L,v}^\times) (1+\frakp_v^r)) \times \langle \varpi_v 
\rangle / \langle \varpi_v^2 \rangle$.  Thus $\mu_v$ can be unramified
or any ramified quadratic character $\eta_{E_v/F_v}$ such that
$N(\frako_{E,v}^\times) \supset N(\frako_{L,v}^\times) (1+\frakp_v^r)$.
Now assume $r=1$.  It is well known that 
$N(\frako_{E,v}^\times) \supset 1+\frakp_v$ if and only if $v$ is odd, and
thus if $v$ is even $\mu_v$ must be unramified.
If $v$ is odd, then $\calO_2(L_v) \simeq \calO_2(E_v)$, so we may take 
$L_v \simeq E_v$ to see that $\mu_v$ can be any quadratic
character.
\end{proof}

\begin{rem}
Let $\calO$ be as in \cref{EisO1}.  Since the 1-dimensional automorphic representations
occurring in $M(\calO)$ form a group, the above lemma yields a complete
description of $\Eis(\calO)$.  
Namely, if $\omega = \mu^2$ for some $\mu \in \widehat{\Cl}^+(N(\hat \calO))$, then 
$\Eis(\calO, \omega) =
\{ (\mu \circ N) \phi : \phi \in \Eis(\calO, \one) \}$.
Otherwise $\Eis(\calO, \omega) = 0$.  
\end{rem}

We define the cuspidal subspaces $S(\calO)$ and $S(\calO, \omega)$
of $M(\calO)$ and $M(\calO, \omega)$ to be the orthogonal complements of
the Eisenstein subspaces.  Hence we have decompositions
\[ S(\calO) = \bigoplus \pi^{\hat \calO^\times}, \quad
S(\calO, \omega) = \bigoplus \pi^{\hat \calO^\times}, \]
where $\pi$ respectively runs over cuspidal representations occurring in
$M(\calO)$ and $M(\calO,\omega)$.

\subsection{Brandt matrices and representation numbers} \label{sec:brandt}
We may realize $M(\calO)$ as $\C^n$ via $\phi \mapsto (\phi(x_1), \dots, \phi(x_n))$.
Let $[\phi]$ denote the column vector ${}^t (\phi(x_1) \dots \phi(x_n))$. 
Then we can define a Hecke action
in terms of Brandt matrices as in \cite{DV} and \cite{me:basis}.

For a nonzero integral ideal $\frakm$, we define the $n \times n$
Brandt matrix $A_\frakm = (a_{ij}(\frakm))$ via
\begin{equation}
a_{ij}(\frakm) = \# \left( \{ \gamma \in \calI_i \calI_j^{-1} : N(\gamma) \frako = \frakm
N(\calI_i \calI_j^{-1}) \} / \calO_\ell(\calI_j)^\times \right).
\end{equation}
Define the Hecke operator $T_\frakm : M(\calO) \to M(\calO)$
via matrix multiplication: $[T_\frakm \phi] = A_\frakm [\phi]$.
We can also express
\[ (T_\frakm \phi)(x) = \sum \phi(x \beta), \]
where $\beta$ runs over the integral right $\calO$-ideals of norm $\frakm$.
Note $T_\frako$ is the identity operator.
The collection of $T_\frakm$'s is a commuting family of real self-adjoint operators
on $M(\calO)$.

Under a class number 1 assumption, the Brandt matrix entries can be
expressed as classical representation numbers of quadratic forms as follows.  
For a (full) $\frako_F$-lattice $\Lambda \subset B$, let $\calO_r(\Lambda)
= \{ \alpha \in B : \Lambda \alpha \subset \Lambda \}$ and
$\calO_r(\Lambda)^1$ be the subset of norm 1 elements in $\calO_r(\Lambda)$.  
Let $\frako_+$ be the subset of totally
positive elements of $\frako$, and $\frako_+^\times$ be the totally positive
units.
For $y \in \frako_+$, define the representation number
\begin{equation}
 r_\Lambda(y) = \# \{ \lambda \in \Lambda : N(\lambda) = y \}.
\end{equation}

Now suppose $h_F^+ = 1$.    Since $h_F = h_F^+$, every element of $\frako_+^\times$ 
is a square, and thus $N_{B/F}(\frako^\times) = \frako_+^\times$.  Fix $y \in \frako_+$.   Then $N(\lambda)\frako = y \frako$ is equivalent
to $N(\lambda) \in y \frako_+^\times$.  Since $\calO_r(\Lambda)^\times =
\frako^\times \calO_r(\Lambda)^1$, we see there is a bijection of sets, 
\[  \{ \lambda \in \Lambda : N(\lambda)\frak \in y \frako_+^\times \} / \calO_r(\Lambda)^\times \simeq \{ \lambda \in \Lambda : N(\lambda)\frak = y \} / \calO_r(\Lambda)^1. \]
Put $\Lambda_{ij} = \calI_i \calI_j^{-1}$ and let $\alpha_{ij}$ be a totally
positive generator of $N(\calI_i \calI_j^{-1})$.  Note 
$\calO_r(\Lambda_{ij}) = \calO_\ell(\calI_j)$.  Also 
$\frako^\times \cap \calO_r(\Lambda_{ij})^1 = \{ \pm 1 \}$, which implies
$\# \calO_r(\Lambda_{ij})^1  = 2w_j$.  Hence, assuming $h_F^+=1$ and
$y \in \frako_+$, we can rewrite our Brandt matrix entries as
\begin{equation}
 a_{ij}(y\frako) = \frac 1{2w_j} r_{\Lambda_{ij}}(y \alpha_{ij}).
\end{equation}
Note that when $i=j$, we get $\Lambda_{ii} = \calO_\ell(\calI_i)$, so
the diagonal Brandt matrix entries are simply
\begin{equation} \label{brandt-diagonal}
a_{ii}(y\frako) = \frac 1{2w_i} r_{\calO_\ell(\calI_i)}(y).
\end{equation}

\subsection{Jacquet--Langlands correspondence}
The Jacquet--Langlands correspondence is a dictionary between
automorphic representations of $B^\times$ and $\GL(2)$.  Here we
present a refinement of it at the level of modular forms
from \cite{me:basis}.

Let $\calO \subset B$ be a special order of level type $(\frakN_1, \frakN_2,
\frakM)$ and write $\frakN = \frakN_1 \frakN_2 \frakM$. 
Let $S_{\two}(\frakN, \omega)$ denote the space of adelic holomorphic
Hilbert cusp forms of level $\frakN$, parallel weight 2, and central
character $\omega$.  (By level $\frakN$, we mean ``level $W(\frakN)$,'' which
is Shimura's analogue of $\Gamma_0(N)$---see, e.g., \cite[Section 5.1]{me:basis}.)
For relatively prime integral ideals $\fraka, \frakb, \frakc$ of $\frako$
such that $\fraka | \frakN_1$ and $\frakb \frakc | \frakN_2$,
let $S_{\two}^{[\fraka, \frakb, \frakc]}(\fraka \frakb \frakc \frakM, \omega)$ be the 
subspace of $S_\two(\frakN, \omega)$ generated by eigenforms which are
$\frakp$-new for each $\frakp | \fraka \frakb \frakc$ and whose corresponding
local representations of $\GL_2(F_v)$ are:
(i) local discrete series for $v | \fraka \frakb \frakc$, 
(ii) supercuspidal for $v | \frakb$,
and (iii) special (i.e., twisted Steinberg) for $v | \frakc$.

For an integral ideal $\mathfrak A \subset \frako$, let $\mathcal H^{\mathfrak A}$
be the Hecke algebra generated by the $T_\frakp$ for $\frakp \nmid {\mathfrak A}$.  
(For our purposes, it does not matter if we work with 
Hecke algebras over $\Z$ or $\Q$ or $\C$.)

\begin{thm} [\cite{me:basis}] \label{JL}
Assume $\ord_\frakp(\frakN_2) = 2$ for each $\frakp | \frakN_2$.  Let $\frakN' = \frakN {\mathfrak D}_B^{-1}$.
We have the following isomorphism of $\mathcal H^{\frakN'}$-modules:
\[ S(\calO, \omega) \simeq \bigoplus 
2^{\# \{ \frakp | \frakb \} } S_{\two}^{[\fraka, \frakb, \frakc]}
(\fraka \frakb \frakc \frakM, \omega), \]
where (i) $\fraka$ runs over divisors of $\frakN_1$ such that
$v_\frakp(\fraka)$ is odd for all $\frakp | \frakN_1$, and (ii)
$\frakb, \frakc$ run over relatively prime divisors of $\frakN_2$
such that $\frakp | \frakb \frakc$ for each $\frakp | \frakN_2$.
\end{thm}

\begin{proof} The result \cite[Corollary 5.5]{me:basis} shows the
stated isomorphism holds for $\mathcal H^{\frakN}$-modules. 
Note $\mathcal H^{\frakN'}$ is the algebra generated by
 $\mathcal H^{\frakN}$ and the $T_\frakp$ for $\frakp | \frakN_1$
 such that $\ord_\frakp(\frakN_1) = 1$.  Since any eigenform 
 occurring in either side of the isomorphism 
 is $\frakp$-new for $\frakp$ sharply
 dividing  $\frakN_1$, the above isomorphism also preserves
 the action of each such $T_\frakp$ (e.g., see \cite{me:cong2} or
 \cite[Remark 5.2]{me:basis}).
\end{proof}

We remark that \cite{me:basis} also establishes an analogous Hecke isomorphism
for \emph{newspaces} without assuming $\ord_\frakp(\frakN_2) = 2$ when 
$\frakp | \frakN_2$.  The difficulty in removing this assumption in the above theorem 
boils down to an issue in the local determination of oldspaces.

\subsection{Average values of quaternionic modular forms} 
The following simple linear-algebraic relation forms the basis of
our calculation of period averages in the next section.

\begin{prop} \label{col-orthog}
Let $\calO$ be an arbitrary order in $B$, and $T$ a real self-adjoint
operator on $M(\calO)$.  Let $(a_{ij})$ be the matrix of $T$ with respect to the
basis $\{e_i \}$, where $e_i \in M(\calO)$ is the indicator function of $x_i \in \Cl(\calO)$.
Let $\Phi$ be an orthogonal basis for $M(\calO)$.
Then, for $1 \le i, j \le n$, we have
\[ \sum_{\phi \in \Phi} \frac{(T\phi)(x_i) \overline{\phi(x_j)}}{(\phi,\phi)}
=  w_j a_{ij}. \]
\end{prop}

\begin{proof}   
The left hand side of the equation above equals
\[ \sum_{\phi \in \Phi} \frac{(T \phi, w_i e_i)\overline{(\phi, w_j e_j)}}{(\phi, \phi)} =
  \sum_{\phi \in \Phi} \frac{(\phi, T w_i e_i)\overline{(\phi, w_j e_j)}}{(\phi, \phi)} 
  = \overline{(T w_i e_i, w_j e_j)} = w_i a_{ji} = w_j a_{ij}. \]
\end{proof}

\begin{rem} In the special case $T=T_\frako$ is the identity, we get
\begin{equation} \label{eq:col-orthog}
\sum_{\phi \in \Phi} \frac{\phi(x_i) \overline{\phi(x_j)}}{(\phi,\phi)}
= \delta_{ij} w_i.
\end{equation}
One can think of quaternionic eigenforms, which are functions on
finite class sets associated to irreducible representations, as a kind of
 analogue of characters of
finite groups.  With this point of view, we can think of \eqref{eq:col-orthog}
as an analogue of column orthogonality for characters 
of finite groups.  Here is a proof of \eqref{eq:col-orthog} 
which is perhaps more parallel to a typical proof of column orthogonality:
 
Write $\Phi = \{ \phi_1, \dots, \phi_n \}$.
Let $A = (a_{ij})$ be the $n \times n$ matrix with $a_{ij} = \frac{\phi_i(x_j)}{\sqrt{w_j(\phi_i,
\phi_i)}}$.  Then $A {}^t \bar A = I$, and \eqref{eq:col-orthog}
follows from orthogonality of the columns of $A$.
\end{rem}

\section{Periods and embeddings} \label{sec:per}

In this section, we study periods of quaternionic
modular forms and their averages.  As in the previous section,
we will often work with arbitrary orders $\calO \subset B$,
but in certain cases restrict to special orders.

\subsection{Periods and averages} \label{sec:41}
Let $\calO$ be a $\frako$-order of $B$, and $K/F$ a quadratic subfield of $B$
such that $\frako_K \subset \calO$.
Define the ideal class map $t \mapsto x(t)$ from $\Cl(K)$ to $\Cl(\calO)$
by $K^\times t \hat \frako_K^\times \mapsto B^\times t \hat \calO^\times$.
This map depends on the specific embedding of $K$ into $B$.
Let $h_K$ be the class number of $K$.

Let $\Pic_\C(\Cl(\calO))$ be the set of formal $\C$-linear combinations of
elements of $\Cl(\calO)$.  For $c = \sum  c_i x_i, d = \sum d_i x_i \in \Pic_\C(\calO)$,
define the height pairing
\[ \langle c, d \rangle = \sum_{i=1}^n w_i c_i \overline{d_i}, \]
and let $\check c \in M(\calO)$ be the dual element given by $\check c(x_i) = w_i c_i$
for $1 \le i \le n$.
Then $c \mapsto \check c$ defines an isometry of $\Pic_\C(\Cl(\calO))$ with
$M(\calO)$.

For $\chi \in \widehat{\Cl}(K)$, consider the complex divisor 
$c_{K,\chi} = \sum \chi(t) x(t) \in \Pic_\C(\Cl(\calO))$, 
which we also denote by $c_K$ when $\chi = 1$.  
For $\phi \in M(\calO)$, we define the period along $K$ against $\chi$ to be 
\begin{equation} \label{per-def}
 P_{K,\chi}(\phi) = (\phi, \check c_{K,\chi}) =  \sum_{t \in \Cl(K)} \phi(x(t)) \chi^{-1}(t). 
\end{equation}
We also denote this by $P_K(\phi)$ when $\chi = 1$.
Define the ideal class embedding numbers
 $h_i = h_{K,i} = \# \{ t \in \Cl(K) : x(t) = x_i \}$ for $1 \le i \le n$.  
 Thus $\sum h_i = h_K$,
and $P_K(\phi) = \sum h_i \phi(x_i)$.

We consider averages of absolute squares of periods in two directions.
First, note that orthogonality of characters of $\Cl(K)$ implies that 
\begin{equation} \label{eq:cl-avg}
\sum_{\chi \in \widehat{\Cl}(K)} | P_{K,\chi}(\phi) |^2 = h_K \sum_{t \in \Cl(K)} |\phi(x(t))|^2
= h_K \sum_{i=1}^n h_i |\phi(x_i)|^2. 
\end{equation}
This expression was used in \cite{MV} to study such averages asymptotically
(for simplicity, restricted to prime level over $\Q$).

Second, for $T = (a_{ij})$ a real self-adjoint operator on $M(\calO)$ and
$\Phi$ an orthogonal basis for $M(\calO)$, we have
\begin{equation} \label{eq:phi-avg}
\sum_{\phi \in \Phi} \frac{P_{K,\chi}(T\phi) \overline{P_{K,\chi}(\phi)}}{(\phi,\phi)} =
\sum_{\phi \in \Phi} \frac{(\phi, T\check c_{K,\chi})(\check c_{K,\chi}, \phi)}{
(\phi, \phi)} = (T \check c_{K,\chi}, \check c_{K,\chi}).
\end{equation}
This expression (in terms of the height pairing) was used
for the exact average formula in \cite{MR}.

Let $\frakN' = \frakN {\mathfrak D}^{-1}_B$, where $\frakN$ is the (reduced)
discriminant of $\calO$.  Thus $\frakp | \frakN'$ if and only if $\calO_\frakp$ is not
a maximal order. Then for a nonzero integral ideal
$\frakm \subset \frako$ which is coprime to $\frakN'$, and for any $\pi$
occurring in $M(\calO)$, $T_\frakm$ acts on $\pi^{\hat \calO^\times}$ by a
scalar $\lambda_\frakm(\pi)$ since the local space $\pi_v^{\calO_v^\times}$
of invariants is 1-dimensional for $v | \frakm$.  For $\phi \in \pi^{\hat \calO^\times}$, set
$\lambda_\frakm(\phi) = \lambda_\frakm(\pi)$.

\begin{prop} \label{prop:per-avg}
Let $\Phi$ be an orthogonal basis of eigenforms for $M(\calO)$.  For a nonzero
integral ideal $\frakm \subset \frako$ coprime to $\frakN'$, consider the divisor $a(\frakm) =\sum_{i=1}^n a_{ii}(\frakm) x_i$ associated to the
diagonal of the Brandt matrix $A_\frakm$.  Then we have
\begin{equation} \label{eq:doub-per}
 \sum_{\phi \in \Phi} \sum_{\chi \in \widehat{\Cl}(K)} \frac{ \lambda_\frakm(\phi) | P_{K,\chi}(\phi) |^2}{(\phi,\phi)} 
 = h_K \langle c_K, a(\frakm) \rangle =
 h_K \sum_{i=1}^n w_i h_i a_{ii}(\frakm).
\end{equation}
\end{prop}

\begin{proof} 
This follows from applying \eqref{eq:cl-avg} and then
\cref{col-orthog} with $i=j$ to the left hand side.  
\end{proof}

For $\frakm$ as above, define the degree of $\deg T_\frakm$ on $M(\calO)$
to be the number of integral right $\calO$-ideals of norm $\frakm$.  Note
each row of the Brandt matrix $A_\frakm$ sums to $\deg T_\frakm$.  For a prime
$\frakp \nmid \frakN$, note $\deg T_\frakp = q_\frakp+1$, where 
$q_\frakp = \#(\frako/\frakp)$.

\begin{cor} \label{cor:per-avg}
Let $\Phi_0$ be an orthogonal basis of eigenforms for $S(\calO)$.
If $\frakm$ is coprime to $\frakN'$,
then
\[ \sum_{\phi \in \Phi_0} \sum_{\chi \in \widehat{\Cl}(K)} \frac{ \lambda_\frakm(\phi) | P_{K,\chi}(\phi) |^2}{(\phi,\phi)} 
=h_K \langle c_K, a(\frakm) \rangle - \delta^+(\calO,\frakm)\deg T_\frakm \frac{h_K^2 |\Cl^+(N(\hat \calO))|}{m(\calO)}, \]
where $\delta^+(\calO,\frakm)$ is $1$ if the class of 
$\frakm$ is trivial in $\Cl^+(N(\hat \calO))$ and $0$ otherwise. 
\end{cor}

\begin{proof} 
Recall that a basis of $\Eis(\calO)$ is given by $\{ \mu \circ N \}$, where
$\mu$ runs over the characters of $\Cl^+(N(\hat \calO))$.
For $\phi = \mu \circ N$, we see that
$P_{K,\chi}(\phi) = \sum_{t \in \Cl(K)} \mu(N_{K/F}(t))\chi^{-1}(t)$
is $h_K$ if $(\mu \circ N)|_{\Cl(K)} = \chi$ and 0 otherwise.
Since each such $\phi$ has norm $(\phi, \phi) = m(\calO)$ and
$\lambda_\frakm(\mu \circ N) = \mu(\frakm) \lambda_\frakm(\one) =
\mu(\frakm) \deg T_\frakm$, we see that
the Eisenstein contribution to \eqref{eq:doub-per} is
\[ \sum_\mu \lambda_\frakm (\mu \circ N) \frac{h_K^2}{m(\calO)} = \frac{h_K^2}{m(\calO)} \sum_\mu \deg T_\frakm \, \mu(\frakm), \] 
where $\mu$ runs over the characters of $\Cl^+(N(\hat \calO))$.
\end{proof}

This double average formula simplifies in various situations.  We explicate two now.
We call $\calO$ \emph{balanced} if $w_i = 1$ for all $1 \le i \le n$, i.e., if
$n = m(\calO)$.
(Balanced implies that each ideal class in $\Cl(\calO)$ has the same volume with respect to a Haar measure, though it is not exactly equivalent.)  In \cref{sec:bal}, we give
elementary criteria for orders to be balanced.

\begin{cor} \label{cor:per-avg-bal}
Let $\Phi_0$ be an orthogonal basis of eigenforms
for $S(\calO)$.  Then:

\begin{enumerate}[(i)]
\item If $\calO$ is balanced, then
\[ \sum_{\phi \in \Phi_0} \sum_{\chi \in \widehat{\Cl}(K)} \frac{| P_{K,\chi}(\phi) |^2}{(\phi,\phi)} 
=h_K^2 \left( 1  -  \frac{|\Cl^+(N(\hat \calO))|}{m(\calO)} \right). \] 

\item
If $h_F^+ = 1$, $\calO$ is special of unramified quadratic type, and 
$\frakm$ is coprime to $\frakN'$,
then
\[ \sum_{\phi \in \Phi_0} \sum_{\chi \in \widehat{\Cl}(K)} \frac{ \lambda_\frakm(\phi) | P_{K,\chi}(\phi) |^2}{(\phi,\phi)} 
=h_K \sum_{i=1}^n w_i h_i a_{ii}(\frakm) - \deg T_\frakm \frac{h_K^2}{m(\calO)}. \]
\end{enumerate}
\end{cor}

\subsection{Exact averages and stability}

Here we will use the double average formula in the previous section
to obtain some precise results about the single averages \eqref{eq:phi-avg}
for fixed $\chi$.  We continue the notation of the
previous section.  In particular, $\frako_K \subset \calO$, where
$\calO$ is an arbitrary order in $B$.

\begin{lemma} \label{lem:stab1}
We have $\langle c_{K,\chi}, c_{K,\chi} \rangle \le
\langle c_{K}, c_{K} \rangle$ for all $\chi \in \widehat{\Cl}(K)$.
Further, $\langle c_{K,\chi}, c_{K,\chi} \rangle = \langle c_{K}, c_{K} \rangle$
for all $\chi$ if and only if $h_i \le 1$ for all $1 \le i \le n$.
\end{lemma}

\begin{proof}
Note that 
\[ \langle c_{K,\chi}, c_{K,\chi} \rangle = \sum_{i=1}^n w_i \, \Bigl \lvert \sum_{x(t) = x_i} \chi(t) \Bigr \rvert^2, \]
where in the inner sum $t$ runs over the elements of $\Cl(K)$ which map to $x_i$ via
the ideal class map.  Clearly this is maximized for $\chi$ being the trivial character,
and we have equality among all $\chi$ when $h_i = \# \{ t \in \Cl(K) : x(t) = x_i \} \le 1$ for all $i$.
Conversely, if $h_i > 1$ for some $i$, then there exist distinct $t, t' \in \Cl(K)$
such that $x(t) = x(t') = x_i$.  Now $\chi(t) \ne \chi(t')$ for some 
$\chi \in \widehat{\Cl}(K)$, and we obtain a strict inequality for such $\chi$.
\end{proof}

We say that the pair $(\calO, K)$ lies in the \emph{semistable range} if $h_i \le 1$ for all $1 \le i \le n$, and in the \emph{stable range} if in addition $h_i = 1$ implies
$w_i = u_K := [\frako_K^\times : \frako^\times]$. Note that if $h_i > 0$, then
$\frako_K \subset \calO_\ell(\calI_i)$, and thus $u_K | w_i$.  If $\calO$ is balanced,
then the stable range is the same as the semistable range.

\begin{cor} \label{cor:semistable}
Let $\Phi_0$ be an orthogonal basis of eigenforms for $S(\calO)$ and $\chi \in
\widehat{\Cl}(K)$.  Then we have the bounds
\[ \sum_{i=1}^n w_i h_i - m^+(\calO, 1) \frac{h_K^2}{m(\calO)} \le 
\sum_{\phi \in \Phi_0} \frac{|P_{K}(\phi)|^2}{(\phi,\phi)} \le 
h_K \sum_{i=1}^n w_i h_i - m^+(\calO, 1)  \frac{h_K^2}{m(\calO)},
\]
where $m^+(\calO,\chi)$ is the number of characters $\mu$ of $\widehat{\Cl}{}^+(N(\hat \calO))$
such that $\mu \circ N_{K/F} = \chi$.
The lower bound is an equality if and only if $(\calO, K)$ in the semistable range, in which
case we have
\[ \sum_{\phi \in \Phi_0} \frac{|P_{K,\chi}(\phi)|^2}{(\phi,\phi)} =
 \sum_{i=1}^n w_i h_i - m^+(\calO, \chi) \frac{h_K^2}{m(\calO)} \]
 for all $\chi \in \widehat{\Cl}(K)$.  
\end{cor}

\begin{proof} We extend $\Phi_0$ to an orthogonal basis $\Phi$ of
$M(\calO)$ by adjoining $\mu \circ N$ for each $\mu \in \widehat{\Cl}{}^+(N(\hat \calO))$.
Then comparing \cref{prop:per-avg} (with $\frakm = \frako$),
\eqref{eq:phi-avg} (for $T=T_\frako$ the identity) and \cref{lem:stab1}, we see that
 \[ \sum_{i=1}^n w_i h_i  \le 
\sum_{\phi \in \Phi} \frac{|P_{K}(\phi)|^2}{(\phi,\phi)} \le 
h_K \sum_{i=1}^n w_i h_i, \]
with the lower bound being equality if and only if we are in the semistable range,
in which case \eqref{eq:phi-avg} for $T=T_\frako$ is independent of $\chi$.
Now we simply subtract of the Eisenstein contribution as in
\cref{cor:per-avg}.
\end{proof}

Note that if $h_F^+ = 1$ and $\calO$ is a special order of unramified quadratic type,
$m^+(\calO, \chi)$ is simply 1 if $\chi = 1$ and 0 otherwise.  If in addition we are
in the stable range, then we simply get
 \begin{equation}
  \sum_{\phi \in \Phi_0} \frac{|P_{K,\chi}(\phi)|^2}{(\phi,\phi)} =
h_K\left(u_K - \delta_{\chi,1} \frac{h_K}{m(\calO)} \right), 
\end{equation}
where $\delta_{\chi, \chi'}$ is the Kronecker delta.  This is always nonnegative
because being in the stable range implies $u_K m(\calO) \ge h_K$.

In general, we can make the above bounds more elementary by observing that
\begin{equation}
h_K \le \sum_{i=1}^n w_i h_i \le w_B h_K,
\end{equation}
where $w_B = \max [\calO_B^\times : \frako^\times]$, where $\calO_B$
runs over maximal orders of $B$.  If $F=\Q$ and $D_B > 3$, then
$w_B \le 3$ and can be determined by congruence conditions.

Finally, we remark the following inheritance properties for suborders.
\begin{lemma} \label{sub-prop}
Let $\calO \subset B$ be an order, and $\calO'$ be a suborder.
If $\calO$ is balanced, then so is $\calO'$.
Similarly if $\frako_K \subset \calO'$, and $(\calO,K)$ is in the (semi-)stable
range then so is $(\calO',K)$.
\end{lemma}

\begin{proof}
We can write $\Cl(\calO') = \{ y_{ij} \}$, where as double cosets in $\hat B^\times$
we have $x_i = \bigsqcup_j y_{ij}$ for all $1 \le i \le n$.  Then each
$\calO_\ell(y_{ij}) \subset \calO_\ell(x_i)$, so each $[ \calO_\ell(y_{ij})^\times : \frako^\times] \le w_i$.  In addition if $t \mapsto y(t)$ denotes the corresponding
class map $\Cl(K) \to \Cl(\calO')$, then we see $\sum_j h'_{ij} = h_i$,
where $h'_{ij} = \# \{ t \in \Cl(K) : y(t) = y_{ij} \}$.
\end{proof}

\subsection{Integral embeddings} \label{sec:emb}

Recall that a quadratic field extension $K/F$ embeds in $B$ if and only if
$K_v/F_v$ is non-split at each $v \in \Ram(B)$ and each $v | \infty$.  
Here we recall some
facts about embedding $\frako_K$ into a quaternionic order $\calO$.

Consider an $\frako$-order $\calO \subset B$.
Let $\Emb(\frako_K, \calO)$ denote the set of embeddings of 
$\frako_K$ into $\calO$ up to conjugation by $\calO^\times$.  Similarly,
define $\Emb(\frako_{K,v}, \calO_v)$ for $v < \infty$.
We have the following formula (see \cite{brzezinski:class} or \cite[Theorem 30.7.3]{voight:book}):
\begin{equation} \label{eq:emb}
\sum_{i=1}^n \# \Emb(\frako_K, \calO_\ell(\calI_i)) = h_K \prod_{v | \frakN}
\# \Emb(\frako_{K,v}, \calO_v). 
\end{equation}
On the left, $\calO_\ell(\calI_i)$ ranges over all isomorphism classes of orders
in the genus of $\calO$ (with multiplicity).  

Now suppose $\calO$ is special of level type $(\frakN_1, \frakN_2, \frakM)$.
In particular, we see that $\frako_K$
embeds into a special order in the genus of $\calO$
 if and only if $\frako_{K,v}$ locally embeds into $\calO_v$ for all $v | \frakN =
 \frakN_1 \frakN_2 \frakM$ and $K/F$ is totally imaginary.
 
Local embedding numbers have been computed for Eichler orders in
\cite{hijikata} and for special orders in \cite{HPS} (see also
\cite[Chapter 30]{voight:book}).  The precise description
is complicated in general, and we will only state it under
 the following assumptions:
 \begin{equation} \label{lev-ass}
 \ord_v(\frakN_2) = 2 \text{ for all } v | \frakN_2 \text{ and }
 \frakM \text{ is squarefree}.
 \end{equation}
We recall this implies that all special orders of level type
$(\frakN_1, \frakN_2, \frakM)$ lie in the same genus.

\begin{lemma} \label{lem:loc-emb}
Assume $\calO$ is a special order type
$(\frakN_1, \frakN_2, \frakM)$ such that \eqref{lev-ass} holds, 
and that $K/F$ is a quadratic field which embeds in $B$.

\begin{enumerate}
\item For $v | \frakN_1$, 
\[ \# \Emb(\frako_{K,v}, \calO_v) =
\begin{cases}
2 & K_v/F_v \text{ unramified}, \\
1 & K_v/F_v \text{ ramified}, \ord_v(\frakN_1) = 1, \\
0 & K_v/F_v \text{ ramified}, \ord_v(\frakN_1) > 1. \\
\end{cases} \]

\item For $v | \frakN_2$,
\[ \# \Emb(\frako_{K,v}, \calO_v) = 
\begin{cases}
0 & K_v/F_v \text{ unramified}, \\
q_v + 1 & K_v/F_v \text{ ramified}.
\end{cases} \]

\item For $v | \frakM$, 
\[ \# \Emb(\frako_{K,v}, \calO_v) =
\begin{cases}
0 & K_v/F_v \text{ unramified}, \\
1 & K_v/F_v \text{ ramified}, \\
2 & K_v/F_v \text{ split}.
\end{cases} \]
\end{enumerate}
\end{lemma}

\begin{proof}
The first two parts follow from \cite[Theorems 5.12 and 5.19]{HPS}
and \cref{lem:HPS}(v).
The third part follows from \cite[Section 2]{hijikata}
(see also \cite[Lemma 30.6.16]{voight:book}).
\end{proof}

Now \eqref{eq:emb} and \cref{lem:loc-emb} immediately yield the following.

\begin{cor} \label{Kconds} Let $\frakN = \frakN_1 \frakN_2 \frakM$ be a nonzero ideal in
$\frako$ such that
$\frakN_1$, $\frakN_2$, $\frakM$ are pairwise coprime, $\# \{ v | \frakN_1 \frakN_2 \} +
[F:\Q]$ is even, and $\ord_v(\frakN_1)$ is odd
for $v | \frakN_1$.  Further assume \eqref{lev-ass}.  Let $K/F$ be a quadratic field
extension.  Then $\frako_K$ embeds in some special order $\calO \subset B$ 
of level type
$(\frakN_1, \frakN_2, \frakM)$ if and only if the following conditions holds:

\begin{enumerate}[(i)]
\item $K/F$ is totally imaginary;

\item $K_v/F_v$ is non-split for each $v | \frakN_1 \frakN_2$;

\item $K_v/F_v$ is unramified for each $v | \frakN_1$ such that $\ord_v(\frakN_1) > 1$;

\item $K_v/F_v$ is ramified for each $v | \frakN_2$;

\item $K_v/F_v$ is either ramified or split for each $v | \frakM$.
\end{enumerate}
\end{cor}

When $(\frakN_1, \frakN_2, \frakM)$ satisfies the hypotheses
in the first 2 sentences of \cref{Kconds}, we will
say $(\frakN_1, \frakN_2, \frakM)$ is a \emph{nice level type}.
When all conditions in this corollary hold (i.e., $\frako_K$ embeds in a
special order of nice level type $(\frakN_1, \frakN_2, \frakM)$) and
the further condition that \emph{(v')} $K_v/F_v$ is split for each $v | \frakM$
also holds, we will say $K/F$ is 
$(\frakN_1, \frakN_2, \frakM)$-\emph{admissible}.
This extra condition  \emph{(v')} will be needed in \cref{sec:54}.

\subsection{Balanced orders} \label{sec:bal}
We can rephase the notion of balanced orders in terms of the existence of certain embeddings.

\begin{lemma} \label{lem:bal}
Let $\eps_1, \dots, \eps_r$ be a set of generators for $\frako_+^\times$.
Then an arbitrary order $\calO$ is balanced if and only if no ring of the form
$\frako[u]$ embeds the genus of $\calO$, where either $u$ is a root of unity of order
$\ge 3$ or $u = \sqrt{-\eps}$ where $\eps = \prod_S \eps_i$ and 
$\emptyset \ne S \subset \{ 1, \dots, r \}$.
\end{lemma}

\begin{proof}
Clearly $\calO$ is unbalanced if and only if there exist an order $\calO' = \calO_\ell(\calI_i)$
in the genus of $\calO$ such that $(\calO')^\times$ contains a unit 
$u \not \in \frako^\times$.  Suppose this is the case.  Then 
$K=F(u)$ is a CM extension in $B/F$.  Let $\mu_K$ denote the roots of unity of $K$.
Hasse's unit index $Q_{K/F} = [ \frako_K^\times: \mu_K \frako^\times ]$ is either 1 or
2.  If $Q_{K/F} = 1$, then $u \in \mu_K$.  Assume $Q_{K/F}=2$.  
Then $u^2 \in \zeta \frako^\times$ for some $\zeta \in \mu_K$.  Hence $\frako[\zeta]$
embeds in $\calO'$, and we may restrict to the case $\zeta = \pm 1$.  Then
$u^2 = -\eps_+$ for some $\eps_+ \in \frako_+^\times$.  
Write $\eps_+ = \prod \eps_i^{d_i}$.  Then, for an appropriate choice of
$\eta \in \frako_+^\times$, one sees $\eps := \eps_+ \eta^2$ is of the
form $\prod_S \eps_i$ for some $\emptyset \ne S \subset \{ 1, \dots, r \}$
By replacing $u$ with $u' = \sqrt{-\eps}$,
we see that  $\frako[u] = \frako[u'] = \frako[\sqrt{-\eps}]$.
This proves the ``if'' direction; the other direction is clear.
\end{proof}

\cref{lem:bal} implies that we can guarantee an order $\calO$ is balanced
by checking a finite number of local conditions, which depend upon $F$.  
Namely, there are finitely many CM extensions $F(u)$ of $F$, with $u$ as in 
\cref{lem:bal}: either $u = \zeta_m$ is a primitive $m$-th root of unity for some 
$m \ge 3$ such that the totally real subfield of $\Q(\zeta_m)$ is contained in $F$, or
$u = \sqrt{-\eps}$ for one of the finitely many $\eps \in \frako^\times_+$ 
of the form above.
For each such $u$, $\frako[u]$ does not embed in the genus of $\calO$ if either
(i) $B/F$ is ramified at some prime where $F(u)/F$ splits, or 
(ii) if $\frako[u] = \frako_K$ where $K=F(u)$ and $\calO$ is a special order
of nice level type $(\frakN_1, \frakN_2, \frakM)$ but
$K$ does not satisfy the conditions of \cref{Kconds}.  
(When $\frako[u] \ne \frako_K$,
one can also use results about optimal embeddings to get similar conditions.)
In particular, if (i) is satisfied for all such $u$, then every order $\calO \subset B$ is
balanced.

We explicate our criteria in the case of $F=\Q$. 

\begin{cor} \label{bal-crit}
Suppose $F=\Q$ and $\calO \subset B$  is a special order of nice level type
$(N_1, N_2, M)$.  Then $\calO$ is balanced if and only if both (1) and (2) below hold.

\begin{enumerate}
\item One of the following holds:

\begin{enumerate}[(a)]
\item there exists a prime $p \equiv 1 \mod 4$ such that $p | N_1$; or

\item there exists a prime $p \equiv 3 \mod 4$ such that $p | M$; or

\item $N_2 \not \in \{ 1, 4 \}$; or

\item $8 | N_1$.
\end{enumerate}

\item And one of the following holds:

\begin{enumerate}[(a)]
\item there exists a prime $p \equiv 1 \mod 3$ such that $p | N_1$; or

\item there exists a prime $p \equiv 2 \mod 3$ such that $p | M$; or

\item $N_2 \not \in \{ 1, 9 \}$; or

\item $27 | N_1$.
\end{enumerate}
\end{enumerate}
\end{cor}

\begin{proof}
By \cref{lem:bal}, $\calO$ is balanced if and only if neither $\Z[i]$ nor
$\Z[\zeta_3]$ embed in the genus of $\calO$.  Now apply \cref{Kconds}.
\end{proof}

In particular, in the setting of this corollary, we see that $\calO$ is balanced if
$N_2 > 9$, or if there exists a $p \equiv 1 \mod 12$
dividing $N_1$, or if there exists a $p \equiv 11 \mod 12$ dividing $M$.

Alternatively, one can treat special orders over $\Q$ 
by comparing the mass formula with the class number formula 
from \cite[Theorem 8.6]{HPS}.

Now we briefly discuss some sufficient conditions for balanced orders $\calO$
when $F=\Q(\sqrt d)$ is real quadratic.  
Assume $d > 1$ is squarefree
and let $\eps$ be a generator for the rank 1 group $\frako_+^\times$.  
Then, assuming $\frako[\eps]$ does not
embed into the genus of $\calO$, \cref{lem:bal} tells us that
$\calO$ is balanced if and only if
no cyclotomic ring $\Z[\zeta_m]$ with $m \ge 3$ embeds in the genus
of $\calO$.  

If $\Z[\zeta_m]$ embeds in the genus of $\calO$, then
$K = K_m := F(\zeta_m)$ embeds in $B$.  The only possibilities
are $K=F \Q(i)$, $K = F\Q(\zeta_3)$ or that $K = \Q(\zeta_m)$ is
cyclotomic of degree 4 containing $F$.  The latter possibility implies
that $m=8$ and $d=2$, $m=12$ and $d=3$, or $m=d=5$.  In these 3 cases,
$\Z[\zeta_m]$ embeds in the genus of $\calO$ if and only if $\frako_K$
does, so we can apply \cref{Kconds}.
When $m = 3, 4$, i.e., $K = F\Q(i)$ or $K=F\Q(\zeta_3)$, it is not
necessarily true that $\frako_K = \frako_F[\zeta_m]$.  
However, if $d > 5$ and $d \equiv 1 \mod 4$, 
then $\Q(i)$, $\Q(\zeta_3)$ and $F$ have pairwise coprime discriminants,
implying that they are pairwise linearly independent, and thus
$\frako_K = \frako_F[\zeta_m]$ for $m=3,4$.  In this situation,
we can again use \cref{Kconds}.

\begin{ex} Let $F=\Q(\sqrt 5)$.  Then $\eps = \frac{\sqrt 5-1}2$ generates
$\frako_+^\times$.  
Let $\frakp_{11} = (\frac{1\pm 3 \sqrt 5}2)$ be a prime of $F$ above $11$.
Both $F(\eps)$ and $\Q(\zeta_5)$ split over $\frakp_{11}$, whereas
both $F\Q(i)$ and $F\Q(\zeta_3)$
split over the prime ideal $(7)$.  Hence if $B$ is any definite quaternion algebra 
ramified at both $(7)$ and $\frakp_{11}$, then any order 
$\calO \subset B$ is balanced. 

Alternatively, $F(\eps)$ and $F\Q(i)$ have rings of integers of the form
$\frako[u]$ where $u=\eps$ or $u = i$, and both are ramified above $(2)$.
Also $F\Q(\zeta_3)/F$ is split above $(2)$.  Hence if $B$ is any definite
quaternion algebra ramified at $(2)$ and also at some prime $\frakp$ lying above
a prime $p$ of $\Q$ that completely splits $\Q(\zeta_5)$ (i.e., $p \equiv 1 \mod 5$),
then any special order $\calO \subset B$ of level type 
$(\frakN_1, \frakN_2, \frakM)$ is balanced if $8 | \frakN_1$.
\end{ex}

\section{Twisted $L$-values} \label{sec:Lvals}

From now on, we assume that $h_F = 1$, $K/F$ is
quadratic, $\frakN = \frakN_1 \frakN_2 \frakM$ 
is a nonzero integral $\frako$-ideal, and that
 $\frakN_1, \frakN_2, \frakM$ and $K/F$ satisfy the conditions in \cref{Kconds}.
 In particular, $(\frakN_1, \frakN_2, \frakM)$ is a nice level type.
From \cref{sec:54} onwards, we will further assume that $K_v/F_v$ is split for 
each $v | \frakM$, i.e., that $K/F$ is $(\frakN_1, \frakN_2, \frakM)$-admissible.

Let $B/F$ be the definite quaternion algebra with discriminant
$\frakD_B = \prod_{\frakp | \frakN_1 \frakN_2} \frakp$.  
Then there exists a special order $\calO \subset B$ of level type 
$(\frakN_1, \frakN_2, \frakM)$ and an embedding of $K$ into $B$
such that $\frako_K \subset \calO$.  We fix such an $\calO$ and such
an embedding,
which defines a class map $\Cl(K) \to \Cl(\calO)$, periods $P_{K,\chi}$,
divisors $c_{K,\chi}$, and ideal class embedding
numbers $h_1, \dots, h_n$ as in the previous section.

For an (irreducible) cuspidal representation $\pi$ of 
$B^\times(\A_F)$ or $\GL_2(\A_F)$, denote by $c(\pi)$ the conductor of $\pi$.
This means that $c(\pi)$ is the ideal in $\frako_F$ corresponding to the exact level
of a newform $\phi \in \pi$.  In particular, $c(\pi)$ is the conductor of the 
Jacquet--Langlands transfer of $\pi$ to $\GL_2(\A_F)$ in the case of 
$B^\times(\A_F)$.  For $v < \infty$, the local
conductor is $c(\pi_v) := \ord_v(c(\pi))$.

\subsection{Local vanishing criteria}

Let $\phi \in M(\calO)$ be a eigenform and 
$\pi$ the associated representation of $B^\times(\A_F)$
(necessarily of trivial central character since $h_F = 1$).
Let $\chi$ be a character of $\Cl(K)$.  It is easy to see that
$P_{K,\chi}(\phi) = 0$ unless the restriction of $\chi$ to $\A_F^\times$
is trivial, so let us assume this.  

Our definition \eqref{per-def}
of the period $P_{K,\chi}$ can also be expressed as an integral
\[ P_{K,\chi}(\phi) = \int_{K^\times \A_F^\times \bs \A_K^\times} \phi(t) \chi^{-1}(t) \, dt, \]
for a suitable choice of Haar measure on $\A_K^\times$.  Via this integral,
$P_{K,\chi}$ extends to a $\chi$-equivariant linear functional on $\pi$.
In other words, $P_{K,\chi}$ is an element of $\Hom_{\A_K^\times}(\pi,\chi)$.
Thus the following gives sufficient criteria for the vanishing of
$P_{K,\chi}(\phi)$, where $\phi \in \pi$.

Denote by $\St_v$ the Steinberg representation of $\GL_2(F_v)$ for $v < \infty$.

\begin{prop} \label{pichi-adm}
Let $\pi$ be a cuspidal representation of $B^\times(\A_F)$
occurring in $S(\calO)$.  Let $\chi$ be a character of $\Cl(K)$
which is trivial on $\A_F^\times$.
The global Hom space  $\Hom_{\A_K^\times}(\pi,\chi)$ 
is nonzero if and only if both of the following hold for all $v | \frakN$ such that
$K_v/F_v$ is ramified:

\begin{enumerate}[(i)]
\item if $v | \frakN_1 \frakN_2$ and $\pi_v$ is $1$-dimensional, 
then $\pi_v(\varpi_{K,v}) = \chi_v(\varpi_{K,v})$;

\item if $v | \frakM$ and $\pi_v \simeq \St_v \otimes \mu_v$, then
$\mu_v(\varpi_v) = - \chi_v(\varpi_{K,v})$.
\end{enumerate}
\end{prop}

\begin{proof}
The global Hom space  $\Hom_{\A_K^\times}(\pi,\chi)$ 
is nonzero if and only if each local Hom space
$\Hom_{K_v^\times}(\pi_v,\chi_v)$ is.  For $v | \infty$,
$\pi_v$ and $\chi_v$ are both trivial, so this local Hom space is clearly nonzero.
Assume from now on $v$ is finite.  
Necessary and sufficient criteria for each local Hom space to be nonzero
were given by Tunnell \cite{tunnell} and Saito \cite{saito}.  In particular, if either $\pi_v$ 
is an (unramified or ramified) principal series or $K_v/F_v$ is split, then necessarily
$B_v$ is split and
$\Hom_{K_v^\times}(\pi_v,\chi_v) \ne 0$ for all $\chi_v: K_v^\times/F_v^\times
\to \C^\times$.

In all other cases, $v | \frakN$ and either $\pi_v$ is a representation of the unit group
$D_v^\times$ of the local 
quaternion division algebra  or $\pi_v$ is a discrete series 
representation of $\GL_2(F_v)$.    In either case, let $\pi_v'$ be the 
Jacquet--Langlands correspondent on the other group.  Furthermore, 
$K_v/F_v$ is a field extension which embeds into $D_v$, and 
we have the dichotomy relation:
\[ \dim \Hom_{K_v^\times}(\pi_v,\chi_v) + 
\dim \Hom_{K_v^\times}(\pi'_v,\chi_v) = 1. \]
Let $\pi_{D,v}$ be the element of $\{ \pi_v, \pi'_v \}$ which is a 
representation of $D_v^\times$.  Then 
$\Hom_{K_v^\times}(\pi_{D,v},\chi_v) \ne 0$ if and only if
$\pi_{D,v} |_{K_v^\times}$ contains $\chi_v$ as a subrepresentation.  
We note that the Tunnell--Saito criterion says that this is the case if and only if
$\eps(\frac 12, \pi_{K,v} \otimes \chi_v) = -1$, though we do not directly
use this in our proof.

First suppose $K_v/F_v$ is unramified.  By our embedding conditions,
this means $v | \frakN_1$, so in particular $\pi_v = \pi_{D,v}$.
Then $K_v^\times = F_v^\times \frako_{K,v}^\times$, so $\chi_v$ occurs
in $\pi_{v}$ if and only if $\pi_{v}$ contains an $\frako_{K,v}^\times$-fixed vector.
In fact $\pi_v^{\calO_v^\times}$ contains an $\frako_{K,v}^\times$-fixed vector
since $\frako_{K,v} \subset \calO_v$.
Thus in this case we always have $\dim \Hom_{K_v^\times}(\pi_v,\chi_v) \ne 0$.

Now assume $K_v/F_v$ is ramified.  By our embedding criteria,
either $\pi_{D,v}$ is 1-dimensional or $v | \frakN_2$ and $\pi_v = \pi_{D,v}$ 
has dimension $> 1$. 
In addition, there are 2 possibilities for $\chi_v$ according to 
$\chi_v(\varpi_{K,v}) = \pm 1$, 
since $K_v^\times / \frako_{K,v}^\times F_v^\times
\simeq \langle \varpi_{K,v} \rangle / \langle \varpi_{K,v}^2 \rangle$.

If $v | \frakN_2$ and $\pi_v = \pi_{D,v}$ has dimension $> 1$,
then $\pi'_v$ is minimal supercuspidal representation of even conductor
(in fact depth 0), so by \cite[Theorem 3.6]{me:basis} 
$\pi_v|_{K_v^\times}$ contains both unramified characters of  $K_v^\times/F_v^\times$.
In particular, $\Hom_{K_v^\times}(\pi_{v},\chi_v) \ne 0$.

Finally we are reduced to the case that $\pi_{D,v} = \mu_v \circ N$ is 1-dimensional,
so $\pi_v' \simeq \St_v \otimes \mu_v$.  It is clear that 
$\Hom_{K_v^\times}(\pi_{D,v},\chi_v) \ne 0$ if and
only if $\mu_v(\varpi_v) = \chi_v(\varpi_{K,v})$.
The proposition now follows from the dichotomy relation by observing that
$\pi_v = \pi_{D,v}$ if $v | \frakN_1 \frakN_2$ and $\pi_v' = \pi_{D,v}$ if $v | \frakM$.
\end{proof}

\subsection{Local $L$-factors}
\label{sec:loc-lfac}
In this section, we describe local $L$-factors in cases that are relevant for us.
Here $v$ denotes a finite place of $F$, $K_v/F_v$ is a quadratic field
extension, $\pi_v$ is a representation of
$\PGL_2(F_v)$, and $\chi_v$ an unramified character of $K_v^\times/F_v^\times$.
We say a supercuspidal $\pi_v$ is unramified dihedral if it is induced from
a regular character of the unramified quadratic extension of $F_v$.  
A supercuspidal $\pi_v$ is unramified dihedral if and only if $\pi_v$ is minimal
of even conductor \cite[Proposition 3.5]{tunnell:1978}.

\begin{lemma} \label{loc-lfac}
We have the following local $L$-factor ratios.
\begin{enumerate}

\item Suppose $\pi_v$ is supercuspidal.  Then
\[ \frac{L(\frac 12, \pi_{K,v} \otimes \chi_v)}{L(1, \pi_v, \Ad)} = 
\begin{cases}
1+q_v^{-1} & \pi_v \text{ unramified dihedral} \\
1 & \pi_v \text{ else}.
\end{cases} \]

\item Suppose $K_v/F_v$ is ramified and $\pi_v \simeq
\St_v \otimes \mu_v$ is a twisted Steinberg representation, where
$\mu_v$ is a quadratic character of $F_v^\times$ such that $\mu_v \circ N_{K_v/F_v}$
is an unramified character of $K_v^\times$ and
$\mu_v(\varpi_{F,v}) = \chi_v(\varpi_{K,v})$.
Then
\[ \frac{L(\frac 12, \pi_{K,v} \otimes \chi_v)}{L(1, \pi_v, \Ad)} = 
 1+q_v^{-1}. \]
\end{enumerate}

\end{lemma}

\begin{proof}
First suppose $\pi_v$ is supercuspidal.  Then $\pi_v$ corresponds to an
irreducible 2-dimensional representation $\sigma_v$ of the Weil group $W_{F_v}$
by the local Langlands correspondence, and the base change $\pi_{K_v}$
corresponds to the restriction of $\sigma_{K_v}$ of $\sigma_v$ to $W_{K_w}$,
where $w$ is a place of $K$ above $v$.  
If $\sigma_{K_v}$ is irreducible, then $\pi_{K_v}$ is supercuspidal.
Otherwise, $\sigma_{K_v} \simeq \tau_w \oplus \tau_w^{-1}$ for some character
$\tau_w$ of $W_{K_w}$.
By inductivity of $L$-functions, $L(s, \tau_w) = L(s, \sigma_v) = L(s, \pi_v) = 1$,
so $L(s, \pi_{K_v}) = L(s, \tau_w) L(s, \tau_w^{-1}) = 1$ and $\pi_{K_v}$ must be
a ramified principal series.

In either case, $\pi_{K_v}$ has conductor at
least 2, so $L(s, \pi_{K,v} \otimes \chi_v) = 1$ since $\chi_v$ is unramified.  
From the calculations in
\cite{NPS}, $L(s,\pi_v,\Ad)$ is either $(1+q_v^{-s})^{-1}$ or $1$ according to whether
$\pi_v$ is unramified dihedral or not.  This yields (1).

Next assume the hypotheses of (2).  Via the local Langlands correspondence,
it is easy to see that $\pi_{K,v} \simeq \St_{K,v} \otimes (\mu_v \circ N)$.
Thus
$L(s, \pi_{K,v} \otimes \chi_v) = L(s, (\mu_v \circ N) \chi_v) = L(s, 1_{F_v})$, 
and $L(s,\pi_v, \Ad) = (1-q_v^{-s-1})^{-1}$.  This yields (2).
\end{proof}

\subsection{Local spectral distributions} \label{sec:loc-spec}
Here we will compute relevant local spectral distributions at a nonarchimedean
place $v$ which splits $K/F$.  To ease notation,
for this section only, all objects are local and we drop the subscripts $v$.
In particular, $F$ is now a $\frakp$-adic field, and $K=F \times F$.

Let $\pi = \pi(\mu, \mu^{-1})$ be
an unramified principal series representation of $\GL_2(F)$ with 
trivial central character.  
Consider a character $\chi$ of $K^\times = F^\times \times F^\times$
given by $\chi(a,b) = \chi_1(a) \chi_{2}(b)$, 
where $\chi_{1}, \chi_{2}$
are 2 unramified characters of $F^\times$ such that $\chi_{2} = \chi_{1}^{-1}$.
Let $\psi$ be an additive character of order 0, and $\mathcal W$ be the
$\psi$-Whittaker model for $\pi$.  For $W \in \mathcal W$,
we abbrieviate $W\bmx a & \\ & 1 \emx$ by $W(a)$.  (Alternatively, one can think of 
$W$ as an element of the Kirillov model.)
Let $d^\times a$ denote the Haar measure on $F^\times$ which gives
$\frako_F^\times$ volume 1, i.e., $d^\times a$ is self-dual with respect to $\psi$.

Put
\[ \ell(W) = \int_{F^\times} W(a) \chi_1^{-1}(a) \, d^\times a. \]
As in \cite[Section 2.1]{MW}, we define the local distribution
\begin{equation}
 \tilde J_{\pi}(f) = \sum_W \ell(\pi(f)W) \overline{\ell(W)}, 
\end{equation}
where $f \in C_c^\infty(\GL_2(F))$ and $W$ runs over an orthonormal basis for 
$\mathcal W$.

If $f_v$ is the characteristic function of
$\GL_2(\frako)$ divided by its volume, then (see \emph{loc.\ cit.})
\begin{equation} \label{eq:52}
 \tilde J_{\pi}(f) =  \frac{L(2,1_F)}{L(1,1_F)}
 \cdot \frac{L(\frac 12, \pi \otimes \chi_1)L(\frac 12, \pi \otimes \chi_2)}{L(1,\pi,\Ad)}.
 \end{equation}
 
\begin{lemma} \label{lem:53} Consider the Eichler order 
$R = \bmx \frako & \frakp \\ \frako & \frako \emx$ of level $\frakp$ in $M_2(\frako)$.
When $f = \vol(R^\times)^{-1} 1_{R^\times}$, we have
\[ \tilde J_{\pi}(f) = \frac{L(2,1_F)}{L(1,1_F)} \left(
 \frac{L(\frac 12, \pi \otimes \chi_1)L(\frac 12, \pi \otimes \chi_2)} {L(1, \pi, \Ad)} 
+ \frac 1{L(1,1_F)} \right). \]
\end{lemma}

\begin{proof}
Let $W_0$ be the Whittaker new vector in $\mathcal W$  of norm 1
which is invariant under
$\GL_2(\frako)$.  Then $\pi(f)$ acts as orthogonal projection onto the
subspace $V$ of $\mathcal W$ spanned by $W_0$ and $W_1 = \pi \bmx \varpi & \\ & 1 \emx W_0$.  From \cite[Section 8B]{FMP}, an orthonormal basis for $V$ is given by
$W_1$ and $W_2$, where
\[ W_2 = \left(\frac{L(1,\pi,\Ad)L(1,1_F)}{L(2,1_F)^2} \right)^{\frac 12} 
\left( W_0 - \frac{\mu(\varpi) + \mu(\varpi)^{-1}}{q^{\frac 12}(1 +q^{-1})} W_1 \right). \]
Hence $\tilde J_{\pi}(f) = |\ell(W_1)|^2 + |\ell(W_2)|^2$.
It is well known that
\[ \ell(W_0) = \left( \frac{L(2, 1_F)}{L(1,\pi, \Ad)L(1,1_F)} \right)^{\frac 12} L(\frac 12, \pi \otimes \chi_1), \]
and we see that $\ell(W_1) = \chi_1(\varpi) \ell(W_0)$
and
\[ |\ell(W_2)|^2 = \lvert1 - \chi_1(\varpi)(\mu(\varpi) + \mu(\varpi)^{-1})q^{-\frac 12}+ q^{-1} \rvert^2 \frac{L(1,\pi,\Ad)}{L(1,1_F)} |\ell(W_0)|^2. \]
This yields the lemma.
\end{proof}

\subsection{$L$-values and periods} \label{sec:54}
For the rest of this section, assume that $K_v/F_v$ is split for each $v | \frakM$,
i.e., $K/F$ is $(\frakN_1, \frakN_2, \frakM)$-admissible.

Suppose $\pi$ is a cuspidal representation of $B^\times(\A_F)$ or $\GL_2(\A_F)$. 
Denote by $S(\pi)$ the set of finite places $v$ at which $\pi$ is ramified and
 $S(\pi, K)$ the subset of $v \in S(\pi)$ such that $K_v/F_v$ is ramified.
Denote by $\eta = \eta_{K/F}$ the quadratic character of $\A_F^\times$ associated
to $K/F$ by class field theory.  Recall $u_K = [\frako_K^\times : \frako_F^\times]$.
Put 
\begin{equation}
 C(K; \frakN) =  2^{\# S(\pi,K)-1} h_F D_F^{-2}
u_K^2 |D_K|^{1/2} \prod_{v | \frakN_2} \frac 1{1+q_v^{-1}}.
\end{equation}
By our assumption that $K_v/F_v$ is split at all $v | \frakM$, we see that $S(\pi,K)$ 
and thus $C(K; \frakN)$ does not actually depend upon $\pi$.

We also set $\Lambda_\frakM(\pi, \chi) = \prod_{v | \frakM} \Lambda_v(\pi, \chi)$,
where
\begin{equation} \label{lambdaM}
 \Lambda_v(\pi, \chi) = 
 \begin{cases}
 \frac 1{1+q_v^{-1}}\left( 1 + (1-q_v^{-1}) \frac{L(1, \pi_v, \Ad)}{L(\frac 12, \pi_{K_v} \otimes \chi)} \right) & v | \frakM, v \not \in S(\pi) \\
1 & v | \frakM, v \in S(\pi) .
  \end{cases}  
\end{equation}

\begin{prop} \label{lval:rel}
Let $\pi$ be a cuspidal representation of $B^\times(\A_F)$ occurring
in $S(\calO)$, and $\Phi_\pi$ be an orthogonal basis for $\pi^{\hat \calO^\times} \subset S(\calO)$.  Let $\chi \in \widehat{\Cl}(K)$ be such that $\Hom_{\A_K^\times}(\pi,\chi) \ne 0$. Then
\begin{equation} \label{eq:lval}
  N(\frakN) \sum_{\phi \in \Phi_\pi} \frac{|P_{K,\chi}(\phi)|^2}{(\phi,\phi)} =
C(K; \frakN) \Lambda_\frakM(\pi,\chi)  \frac{L(\frac 12, \pi_K \otimes \chi)}
{ L(1, \pi, \Ad) }.
\end{equation}
\end{prop}

\begin{proof}
Let $\frakN' = c(\pi)$, and let $\calO'$ be a special order of level type
$(\frakN_1', \frakN_2', \frakM')$ such that $\frakN' = \frakN_1' \frakN_2' \frakM'$
and $\calO' \supset \calO$.

Take for measures on $B^\times(\A_F)$, $\A_K^\times$ and $\A_F^\times$ 
the product of the local Tamagawa measures.  Consider a test function 
$f = \prod f_v \in C_c^\infty(B^\times(\A_F))$ where
$f_v$ is the characteristic function of $(\calO_v')^\times$
divided by its volume for $v < \infty$.  For $v | \infty$,
choose $f_v$ so that $\int_{B^\times(F_v)} f_v(g) \, dg = 1$.  Then
$\pi(f)$ acts as orthogonal projection onto $S(\calO') \cap \pi$.
For $\phi, \phi' \in S(\calO)$, put
\[ (\phi, \phi')_{\mathrm{Tam}} = \int_{\A_F^\times B^\times \bs B^\times(\A_F)} \phi(x)
\overline{\phi'(x)} \, dx. \]
Since $\A_F^\times B^\times \bs B^\times(\A_F)$ has Tamagawa volume 2,
we see $(\phi, \phi')_{\mathrm{Tam}} = \frac 2{m(\calO)}(\phi, \phi')$.

Let $S_0(\pi) = S(\pi,K) \cup \{ v \in S(\pi) : c(\pi_v) \ge 2 \}$.
Then by the proof of \cite[Theorem 4.1]{MW} (see \cite[Theorem 1.1]{FMP} 
for a formulation consistent with our present choice of measures), 
we have the formula
\[ J_\pi(f) = \frac {2^{\# S(\pi,K) - 1}}{\pi^{[F:\Q]}} \sqrt{\frac{D_F}{|D_K|}} L_{S(\pi)}(1,\eta)L^{S(\pi)}(2,1_F)
 \frac{L^{S_0(\pi)}(\frac 12, \pi_K \otimes \chi)}
{ L^{S_0(\pi)}(1, \pi, \Ad) },\]
where $J_\pi(f)$ is a certain spectral distribution, which for our choice of test
function $f$ is given by,
\begin{equation} \label{eq:Jpi}
 J_\pi(f) = \frac{m(\calO)}2 \left( \frac{\vol(K^\times \A_F^\times \bs \A_K^\times)}{h_K} \right)^2 \sum_{\phi} \frac{|P_{K,\chi}(\phi)|^2}{(\phi,\phi)},
\end{equation}
and $\phi$ runs over an orthogonal basis $\Phi'$ for $\pi^{(\hat \calO')^\times}$.
Note
\[
\frac{\vol(K^\times \A_F^\times \bs \A_K^\times)}{h_K} 
= \frac{2L(1, \eta)}{h_K} = \frac {2^{[F:\Q]}\sqrt{D_F}}
{u_K h_F \sqrt{|D_K|}}. 
\]
Putting everything together and using the identity
\[ \frac{L(2, 1_F)}{|\zeta_F(-1)|} = \frac{(2\pi)^{[F:\Q]}}{D_F^{3/2}} \]
shows that 
$N(\frakN) \sum_{\phi \in \Phi'} \frac{|P_{K,\chi}(\phi)|^2}{(\phi,\phi)}$
equals
\[ C(K; \frakN) \cdot  \prod_{v \in S(\pi, K)} (1+q_v^{-1}) 
\cdot \prod_{v | \frakM (\frakM')^{-1}} \frac 1{1+q_v^{-1}} \cdot 
\frac{L^{S_0(\pi)}(\frac 12, \pi_K \otimes \chi)}
{ L^{S_0(\pi)}(1, \pi, \Ad) }. \]
By \cref{loc-lfac}, we may rewrite this as
\begin{equation} \label{eq:Pprime}
N(\frakN) \sum_{\phi \in \Phi'} \frac{|P_{K,\chi}(\phi)|^2}{(\phi,\phi)} =
C(K; \frakN) \cdot \prod_{v |  \frakM (\frakM')^{-1}} \frac 1{1+q_v^{-1}}
\cdot \frac{L(\frac 12, \pi_K \otimes \chi)} { L(1, \pi, \Ad) }.
\end{equation}

Now we need to relate the sum of square periods over $\Phi'$ to one over $\Phi_\pi$.

By \cite{me:basis}, we have $\pi^{\hat \calO^\times} = \pi^{(\hat \calO'')^\times}$,
where $\calO''$ is a special order of level type $(\frakN_1', \frakN_2', \frakM)$
lying between $\calO'$ and $\calO$.  Now we take a test function 
$f' = \prod f_v' \in C_c^\infty(B^\times(\A_F))$ such that $f_v' = f_v$ for $v \nmid
\frakM (\frakM')^{-1}$ and $f'_v$ is the characteristic function of
$(\calO''_v)^\times = \calO_v^\times$ divided by its volume for $v | \frakM (\frakM')^{-1}$.  
Then, as in  \cite[Section 8B]{FMP}, we can write 
\[ J_\pi(f') = J_\pi(f) \cdot \prod_{v | \frakM (\frakM')^{-1}} 
\frac{\tilde J_{\pi_v}(f'_v)} {\tilde J_{\pi_v}(f_v)}, \]
where $J_\pi(f')$ is given by \eqref{eq:Jpi} but now with
$\phi$ running over $\Phi_\pi$.  Now comparing \eqref{eq:52} with \cref{lem:53}
shows that for $v | \frakM (\frakM')^{-1}$, 
\[ \frac{\tilde J_{\pi_v}(f'_v)} {\tilde J_{\pi_v}(f_v)} =
1 + (1-q_v^{-1})\frac{L(1, \pi_v, \Ad)}{L(\frac 12, \pi_{K_v} \otimes \chi)}. \]
Combining this with \eqref{eq:Pprime} gives the proposition.
\end{proof}

\subsection{Average $L$-values}

Let $\mathcal F(\frakN)$ denote the set of holomorphic parallel weight 2
cuspidal automorphic representations $\pi$ of $\GL_2(\A_F)$ with
trivial central character such that (i) $c(\pi) | \frakN$;
(ii) $c(\pi_v)$ is odd for $v | \frakN_1$; and (iii) $\pi_v$ is a discrete
series representation for $v | \frakN_2$.  
For $\pi \in \mathcal F(\frakN)$, let $\pi_B$ be the corresponding automorphic
representation of $B^\times(\A_F)$.
For a character $\chi$ of $\Cl(K)$, denote by $\mathcal F(\frakN; \chi)$
the subset of $\pi \in \mathcal F(\frakN)$ such that 
$\Hom_{\A_K^\times}(\pi_B, \chi) \ne 0$.  Thus by \cref{pichi-adm}, 
$\pi \in \mathcal F(\frakN; \chi)$ if and only if, for any place $v | c(\pi)$ such that 
$K_v/F_v$ is ramified (so $v | \frakN_1 \frakN_2$) and $\pi_v \simeq \St_v \otimes
\mu_v$, we have $\mu_v(\varpi_{F,v}) = \chi_v(\varpi_{K,v})$.

\begin{thm} \label{main-thm}
Let $\frakm$ be a nonzero integral ideal of $\frako_F$ which is
coprime to $\frakN \frakD_B^{-1}$.  Then
\begin{multline*} C(K; \frakN) \sum_{\chi \in \widehat{\Cl}(K)}   \sum_{\pi \in \mathcal F(\frakN; \chi)} \lambda_\frakm(\pi)  \Lambda_\frakM(\pi,\chi)
\frac{L(\frac 12, \pi_K \otimes \chi)}
{L(1, \pi, \Ad)} = \\
N(\frakN) h_K \left( \langle c_K, a(\frakm) \rangle - \delta^+(\calO, \frakm) \deg T_\frakm
\frac{h_K |\Cl^+(N(\hat \calO))|}{m(\frakN_1, \frakN_2, \frakM)} \right).
\end{multline*}
\end{thm}

\begin{proof}
Since the periods $P_{K,\chi}$ automatically vanish on $\pi_B$ if 
$\Hom_{\A_K^\times}(\pi_B, \chi) = 0$, this theorem follows directly
from \cref{lval:rel}, \cref{cor:per-avg} and \cref{JL}.
\end{proof}

\begin{rem}
By the $\eps$-factor criteria of Tunnell and Saito discussed earlier,
note that $\pi \in \mathcal F(\frakN; \chi)$ implies that
$\pi_K \otimes \chi$ has root number $+1$.  (The converse is not typically 
true.)  Thus the restriction to such $\pi$ precludes one
trivial reason for $L(\frac 12, \pi_K \otimes \chi)$ to vanish, the root number
over $K$ being $-1$.  There
is also a trivial way for this central value to vanish if the root number of
$K$ is $+1$---if $L(s, \pi_K \otimes \chi) = L(s, \tau_1) L(s,\tau_2)$ 
factors as two degree 2 $L$-functions over $F$ with root number $-1$.
For instance, if $\chi$ is trivial, then $L(s,\pi_K) = L(s, \pi)L(s, \pi \otimes \eta_{K/F})$
so $L(\frac 12, \pi_K)$ is forced to vanish if $\pi$ has root number $-1$.
See also \cite[Remark 6.10]{me:zeroes} regarding a factorization when 
$\chi$ is quadratic.
\end{rem}

As in \cref{sec:per}, this double average formula simplifies in a variety
of situations, can be used to obtain bounds not involving
$\calO$, and yields exact bounds on single averages.  Under suitable
conditions, we can also pick out the average restricted to newforms.
For simplicity, we only carry this out over $\Q$ below.

\subsection{Consequences over $\Q$} 
We now assume $F=\Q$, and explain how to deduce the double average value 
formula \cref{main-thm:Q} from \cref{main-thm}. 
Write $N, N_*, M, D_B$ in place of
$\frakN, \frakN_*, \frakM, \frakD_B$, where the roman types now
represent positive
generators of their fraktur counterparts.  If $\pi$ is an automorphic
representation corresponding to a weight 2 newform $f$, we also write
the factors $\Lambda_*(\pi,\chi)$ as $\Lambda_*(f,\chi)$.

\begin{proof}[Proof of \cref{main-thm:Q}]
Let $f \in \mathcal F(N; \chi)$ be a newform level $N_f \mid N$
and $\pi$ be the associated cuspidal representation of $\GL_2(\A_\Q)$.
To translate our automorphic $L$-values to classical quantities
first note that by \cite[Theorem 5.1]{hida}, we have
\[ L^{S_2(\pi)}_\fin(1, \pi, \Ad) = \frac{8 \pi^3}{N_f} (f, f), \]
where $(f, f) = (f,f)_{N_f} = \int_{\Gamma_0(N_f) \bs \mathfrak H} |f(x+iy)|^2 \, dx \, dy$ 
denotes the usual Petersson norm on $X_0(N_f)$, and
$S_2(\pi)$ denotes the set of finite primes $p$ where $c(\pi_p) \ge 2$.
From the calculations in the proof of \cref{loc-lfac}, we have
\[ L_\fin(1, \pi, \Ad) = \Lambda_{N_2}(f)^{-1} L^{S_2(\pi)}_\fin(1, \pi, \Ad), \]
where $\Lambda_{N_2}(f) = \prod_{p | N_2} \Lambda_{p}(f)$ and 
\begin{equation} \label{lambdaN2}
\Lambda_p(f) = \Lambda_p(f,\chi) = 
\begin{cases}
{1+p^{-1}} & p | N_2, \, \pi_p \text{ supercuspidal}, \\
{1-p^{-2}} & p | N_2, \, \pi_p \text{ ramified twist of } \St_p, \\
1 & p | N_2, \,  \pi_p \text{ unramified twist of } \St_p.
\end{cases}
\end{equation}
Also note that for $v | \infty$,
\[ \frac{L(\frac 12, \pi_{K,v} \otimes \chi_v)}{L(1,\pi_v, \Ad)} = 2\pi . \]
Thus 
\begin{equation} \label{lval-translate} 
\frac{L(\frac 12, \pi_K \otimes \chi)}{L(1, \pi, \Ad)} =
\frac {N_f}{4\pi^2} \Lambda_{N_2}(f) \frac{L_\fin(\frac 12, f, \chi)}{(f,f)}.
\end{equation}
Now \cref{main-thm:Q} follows from \cref{main-thm} together with \cref{EisO1}.  
\end{proof}

In the above proof, \eqref{lval-translate} means we can rewrite \eqref{eq:lval} as
\begin{equation} \label{eq:lvalQ}
 \sum_{\phi \in \Phi_\pi} \frac{|P_{K,\chi}(\phi)|^2}{(\phi,\phi)} 
= \frac{C(K; N)}{4\pi^2} \Lambda_N(f, \chi)  \frac{N_f}N \frac{L_\fin(\frac 12, f, \chi)}{(f,f)}.
\end{equation}
For some of our remaining results on $L$-value averages, we will use this 
together with the results of \cref{sec:per}.

For explicit calculations for later use, we note that we can express the factors 
$\Lambda_p(f,\chi)$ for $p \mid M$, $p \nmid N_f$ from \eqref{lambdaM} 
in terms of Fourier coefficients $a_p(f)$ as
\begin{equation} \label{lambdaMQ}
 \Lambda_p(f, \chi) = \frac 1{1+p^{-1}} \left( 1 + 
\frac{|1+\chi(\frakp)(\chi(\frakp)-a_p(f))p^{-1}|^2}{1+2p^{-1}+(1-a_p(f)^2)p^{-2}}
\right)
\qquad (\text{for } p \mid M, \, p \nmid N_f),  
\end{equation}
where $\frakp$ denotes a prime of $K$ above $p$.

\begin{lemma} \label{lem:wi}
 Let $\calO \subset B$ be a special order of nice level type
$(N_1, N_2, M)$, and set $N=N_1 N_2 M$.  
Then $\# \calO^\times \le 6$ unless $N \le 4$.
\end{lemma}

\begin{proof}
Suppose $\# \calO^\times > 6$.  Note that if $\calO' \supset \calO$, then $\# (\calO')^\times \ge \# \calO^\times$.
Moreover $\# \calO_B^\times \le 6$ for every maximal order $\calO_B \subset B$
unless $D_B \in \{2, 3 \}$, in which case $\calO_B$ is unique up to isomorphism.
When $D_B = 2$, $\# \calO_B^\times = 24$, and when $D_B = 3$, 
$\# \calO_B^\times = 12$.  In both cases, the $\Z$-order in $B$ generated by
$\calO_B^\times$ is all of $\calO_B$.  Hence the only possibility with $N > 3$ is that
$D_B = 2$ and $\calO$ is a non-maximal order
in $B$.  Assume this.

Then $\calO^\times$ is a proper subgroup of 
$\calO_B^\times \simeq \mathrm{SL}_2(\F_3)$ of order greater than 6, 
and the only possibility is $\# \calO^\times = 8$.
Necessarily $\Z[i]$ embeds in the genus of $\calO$, so by \cref{lem:loc-emb}
we must have that $N_1 N_2$ is 2 or 4.  If $N=4$, then there is a unique order up to
isomorphism and $\# \calO^\times = 8$ does occur.

Suppose $N > 4$.  Then $M > 1$, and by enlarging
$\calO$ if necessary we may suppose $M = p$ and $N=2p$.  
Again by \cref{lem:loc-emb},
we must have $p \equiv 1 \mod 4$.  From the mass formula (\cref{mass-form}), 
$m(\calO) = \frac{p+1}{12}$.
By the class number formula, $h(\calO) = m(\calO) + \frac 76$ or $h(\calO) = m(\calO) +
\frac 12$, depending on whether $p \equiv 1 \mod 3$ or $p \equiv 2 \mod 3$.
Correspondingly $\sum (1 - \frac 1{w_i})$ is either $\frac 76$ or $\frac 12$.  In the
latter case, there must be exactly one $w_i > 1$, and it must be 2, which implies
$\# \calO^\times \le 4$.  So consider the former case.  Then there must be exactly two 
$w_i$'s such that $w_i > 1$, and their values must be 2 and 3.  In
particular, $w_i = 4$ is impossible.
\end{proof}

\begin{lemma}  \label{lem:stab}
Suppose  $D_B > |D_K|$ with $\gcd(D_B, D_K) = 1$.  
Then for any order $\calO \subset B$ containing $\frako_K$,
the pair $(\calO, K)$ lies in the stable range.
\end{lemma}

\begin{proof}
The stable average value formula \cite[Theorem 6.10]{FW} together with
\cref{cor:semistable} implies
that $(\calO_B, K)$ is in the stable range for any maximal order
$\calO_B$ containing $\frako_K$.  Now apply \cref{sub-prop}.
\end{proof}

\begin{proof}[Proof of \cref{thm2}]
We want to relate double averages over $\mathcal F(N)$ to 
double averages of $\mathcal F_0(N)$.  We do this by subtracting away the
contribution of the $N_1N_2$-oldforms using inclusion-exclusion.

For a subset $\Sigma$ of primes dividing $D_B^{-1} N_1 N_2$ at which we want
to lower the level, define $N^\Sigma =
N_1^\Sigma N_2^\Sigma M$ as follows.  First put $N_2^\Sigma = \prod p^2$,
where $p$ runs over primes dividing $N_2$ such that $p \not \in \Sigma$.  
Then set $N_1^\Sigma = \prod p^{\ord_p(N_1)} \prod q^{\ord_p(N_1)-2} \prod r$,
where $p$ runs over the primes dividing $N_1$ not in $\Sigma$, and $q$
(resp.\ $r$) runs over the primes dividing $N_1$ (resp.\ $N_2$) which are in $\Sigma$.

Let
\[  A_0(N_1, N_2) =  \frac{C(K; N)}{4 \pi^2} \sum_{\chi \in \widehat{\Cl}(K)}   \sum_{f \in \mathcal F_0(N; \chi)} \Lambda_N(f,\chi) N_f \frac{L_\fin(\frac 12, f,  \chi)} {(f,f)}, \]
and $A(N_1, N_2)$ denote the corresponding expression where the sum over
$f \in \mathcal F_0(N; \chi)$ is replaced by the sum over $f \in \mathcal F(N; \chi)$.
Then
\[ A_0(N_1, N_2) =  \sum_\Sigma (-1)^{\# \Sigma} \left(\prod_{p | N_2, p \in \Sigma} \frac 1{1+p^{-1}} \right) A(N_1^\Sigma, N_2^\Sigma), \]
where $\Sigma$ runs over all subsets of $\{ p : p | D_B^{-1} N_1 N_2 \}$.
It follows from \cref{sub-prop,lem:wi,lem:stab},
that any of the conditions (1)--(3) in the initial assumption of \cref{thm2} 
imply the equality $\sum h_i = u_K h_K$ for any
order $\calO \supset \frako_K$ of nice level type $(N_1^\Sigma, N_2^\Sigma, M)$.  
Now from \cref{main-thm:Q}, we have
\[ A(N_1^\Sigma, N_2^\Sigma) = N^\Sigma h_K^2 (u_K - 2^{\omega'(N_2^\Sigma)}c(N_1^\Sigma, N_2^\Sigma, M) ). \]
Then the theorem follows from the facts that
\[ \sum_\Sigma (-1)^{\# \Sigma} N^\Sigma \left(\prod_{p | N_2, p \in \Sigma} \frac 1{1+p^{-1}} \right)
= N \prod_{p | N_1'} \left( 1- \frac 1{p^2} \right) \cdot 
\prod_{p | N_2} \frac p{p+1}, \]
and
\[ \sum_\Sigma (-1)^{\# \Sigma} N^\Sigma \left(\prod_{p | N_2, p \in \Sigma} 
\frac 1{1+p^{-1}} \right) {2^{\omega'(N_2^\Sigma)}}{c(N_1^\Sigma, N_2^\Sigma, M)} 
 \]
is 0 unless $N_1' = 1$ and $N_2$ is odd, in which case it is
\[  N{c(N_1, N_2, M)} = 12 \prod_{p | N_1 N_2} \frac 1{1-p^{-1}}
\prod_{p | N_2 M} \frac 1{1+p^{-1}}. \]
\end{proof}

\begin{proof}[Proof of \cref{cor:stable}]
Keeping the notation of the previous proof,
define $A^\chi(N_1)$ and $A_0^\chi(N_1)$
to be the ``$\chi$-parts'' of $A(N_1, 1)$ and $A_0(N_1, 1)$---i.e., 
the analogous expressions with only
a single sum over $f$ in $\mathcal F(N; \chi)$ or $\mathcal F_0(N; \chi)$
for a fixed $\chi$.  Then
\[ A_0^\chi(N_1) = \sum_\Sigma (-1)^{\# \Sigma} A^\chi(N_1^\Sigma) \]
By \eqref{eq:lvalQ} and \cref{cor:semistable}, in the stable range we have
\[ A^\chi(N_1^\Sigma) = N^\Sigma (u_K h_K - \delta_{\chi,1} {h_K^2}{c(N_1^\Sigma,1,M)} ). \]
Then the formula follows by the calculations from the previous proof.
\end{proof}

\subsection{Lower bounds and effective nonvanishing} \label{sec:lowbds}
We keep to the situation of $F=\Q$ and obtain exact lower bounds
for single averages of $L$-values in special situations.  In particular, this will 
give the effective nonvanishing results stated in the introduction.

We first remark that, when $N=N_1 > 4$ is squarefree,  \eqref{eq:lvalQ} and
\cref{cor:semistable} imply
\begin{align} \label{eq:lower-bound1}
 \frac{C(K; N)}{4\pi^2} \sum_{f \in \mathcal F^\new(N)} \frac{L_\fin(\frac 12, f)L_\fin(\frac 12, f \otimes \eta_K)}{(f, f)} &= \sum_{i=1}^n w_i h_i - h_K^2 c(N_1, 1, 1) \\
 &\ge  h_K( u_K - \frac{12 h_K}{\phi(N)}). \nonumber
 \end{align}
 Here $\phi(N)$ is the Euler totient function.
 In particular, we get that $L(\frac 12, f)L(\frac 12, f \otimes \eta_K) \ne 0$ for some
 $f \in S_2^\new(N)$ as soon as $\phi(N) > \frac{12 h_K}{u_K}$. 

\begin{prop} \label{lowbdN}
Fix a prime $p$ and let $r > 1$.  Let $N = N_0 p^r \ge 11$,
where $N_0$ is a squarefree product of an even number of primes (possibly $N_0=1$)
coprime to $p$.

\begin{enumerate}[(i)]
\item
If $r$ is odd, $N \ne 27$ and $K/\Q$ is $(N,1,1)$-admissible, we have
\[  \frac{C(K; N)}{4\pi^2} \sum_{f \in \mathcal F^\new(N)} \frac{L_\fin(\frac 12, f)L_\fin(\frac 12, f \otimes \eta_K)}{(f, f)} \ge h_K (u_K - \frac{3h_K}{p^2} ). \]

\item
If $r = 2$  and $K/\Q$ is $(N_0,p^2,1)$-admissible, we have
\[  \frac{C(K; N)}{4\pi^2} \sum_{f \in \mathcal F^\new(N)} \Lambda_p(f) \frac{L_\fin(\frac 12, f)L_\fin(\frac 12, f \otimes \eta_K)}{(f, f)} \ge h_K (u_K - \frac {3h_K}{p+1} (1 + \frac 4{(p-1)\phi(N_0)} ) ). \]
\end{enumerate}
\end{prop}
 
\begin{proof}
We realize $N=N_1 N_2$ with our running conventions, so $M=1$ and $D_B = N_0 p$.
Let 
\[  A_0(N) = \frac{C(K; N)}{4 \pi^2} \sum_{f \in \mathcal F_0(N; 1_K)}  \Lambda_{N_2}(f) N_f \frac{L_\fin(\frac 12, f,  1_K)} {(f,f)}, \]
and $A(N)$ denote the corresponding sum over $f \in \mathcal F(N; 1_K)$.
(These are the $\chi=1_K$ parts of the sums denoted $A_0(N_1, N_2)$ and $A(N_1, N_2)$ in the previous section.)  Note that $\mathcal F_0(N; 1_K) \subset 
\mathcal F^\new(N)$, so by non-negativity of central $L$-values,
the left hand sides above are at least $\frac {A_0(N)}N$.  Hence it suffices to show the right
hand sides are lower bounds for $\frac{A_0(N)}N$.

We have $A_0(N) = A(N) - \frac 1{1+p^{-1}} A(N/p)$
or $A_0(N) = A(N) - A(N/p^2)$, according to whether $r=2$ or $r$ is odd.
As in the proof of \cref{thm2}, our restrictions on $N$
imply that, for a special order of level $N/p$ or $N/p^2$,
the quantity $\sum w_i h_i \le 3 h_K$.
Now the asserted upper bound on $\frac{A_0(N)}N$ follows from
using \cref{cor:semistable} for a lower bound on $A(N)$ and upper bounds
on $A(N/p)$ and $A(N/p^2)$.
\end{proof}

The $N_0=1$ case of this proposition, combined with \eqref{eq:lower-bound1},
yields \cref{cor:12}, and an analogue for levels $N=p^2$.  E.g., if $K$ is ramified
at $p \ge 7$, then one gets a non-vanishing result in level $p^2$ provided $p + 1 > \frac{5h_K}{u_K}$.

\begin{prop} \label{lowbdM}
Suppose $N_1 > 3$ is squarefree, $N_2 = 1$ and $M=p > 3$.
Assume $K/\Q$ is $(N_1, 1, p)$-admissible.
Then 
\[ \frac{C(K; N)}{4\pi^2} \sum_{f \in \mathcal F^\new(N)} \Lambda_p(f,1_K) \frac{L_\fin(\frac 12, f)L_\fin(\frac 12, f \otimes \eta_K)}{(f, f)} \ge h_K (u_K - \frac{3h_K \Xi(p)}p), \]
where $\Xi(p) := 2 \left( \frac{1+p^{-1/2}}{1-p^{-1}} \right)^2.$
\end{prop}

\begin{proof}
Define quantities $A_0(N)$ and $A(N)$ as in the previous proof.
By Deligne's bound $a_p(f) \le 2 \sqrt p$ and \eqref{lambdaMQ}, we have
\[ \frac 1{1+p^{-1}} \le \Lambda_p(f,1_K) \le \Xi(p) \]
for $f \in \mathcal F(N_1; 1_K)$.
Then use $A_0(N) \ge A(N) - \Xi(p) A(N/p)$ and proceed as before.
\end{proof}

Note $\Xi(p)$ is a decreasing function in $p$.  In particular, $\Xi(p) < 4$ if $p \ge 13$ and
$\Xi(p) < 3$ if $p \ge 31$.  Hence one of the $L$-values appearing on the left
of \cref{lowbdM}
must be nonzero if $p \ge \max \{ \frac{12h_K}{u_K}, 13 \}$ or if
$p \ge \max \{ \frac{9h_K}{u_K}, 31 \}$.

\section{Examples} \label{sec:ex}

In this section, we illustrate our formulas with a few examples when $F=\Q$.  
In each of examples, we numerically computed the relevant $L$-values to
serve as a quality check for our formulas.
We performed our calculations using a combination of Magma \cite{magma}
and Sage \cite{sage}.  
We also used Dan Collins' Sage code to compute
Petersson norms (see \cite{collins}).

\begin{ex} \label{ex:11}
Let $N=D_B = 11$, so $\calO$ is a maximal order.
Then $n=2$ and there are 2 maximal orders up to isomorphism,
one with 4 units and one with 6 units.  Order the ideal classes
$\calI_1, \calI_2$ so that $w_1 = 2$ and $w_2 = 3$.  We may
write $S(\calO) = \langle \phi \rangle$, where $\phi(x_1) = 2$,
$\phi(x_2) = -3$.  Then $\phi$ corresponds to the unique newform
$f \in S_2(11)$.  

Quadratic forms associated to $\calI_1$ and $\calI_2$ are
$Q_1(x,y,z,w) = \frac{(2x+z)^2+(2y+w)^2+11z^2+11w^2}4$ and
$Q_2(x,y,z,w) = \frac{(4x+z+2w)^2+(4y+2z-w)^2+11z^2+11w^2}4$.
By \eqref{brandt-diagonal}, the diagonal Brandt matrix entries are 
$a_{11}(n) = \frac{r_{Q_1}(n)}4$ and $a_{22}(n) = \frac{r_{Q_2}(n)}6$.
Eichler's work on the basis problem shows $a_n(f) = 3a_{22}(n) - 2a_{11}(n) =
\frac{r_{Q_2}(n)-r_{Q_1}(n)}2$ for all $n$.  Comparing this with the trace of the
$p$-th Brandt matrix shows, for $p \ne 11$, we have
$a_p(f) = a_{11}(p) + a_{22}(p) - p - 1 = \frac {5r_{Q_1}(p)-12p-12}8$.

First suppose $K=\Q(i)$.  Then necessarily $h_1 = 1$ and $h_2 = 0$
so $\frac{|P_K(\phi)|^2}{(\phi, \phi)} = \frac 45$.  From \eqref{eq:lvalQ} or 
\cref{main-thm:Q}, we see 
\[ \frac{L_\fin(\frac 12, f)L_\fin(\frac 12, f \otimes \eta_K)}{(f,f)} = \frac{4\pi^2}5. \]
Consider a prime $p \ne 11$.  Then by \cref{main-thm:Q} with $m=p$, we have
\[ a_p(f) = \frac{5}{4\pi^2} \frac{a_p(f) L_\fin(\frac 12, f)L_\fin(\frac 12, f \otimes \eta_K)}{(f,f)} = \frac {5r_{Q_1}(p)-12p-12}8. \]
This independently recovers the above formula for $a_p(f)$ which arose from linear relations of theta series.
 
Next suppose $K=\Q(\sqrt{-11})$ which also has class number 1.  One can check 
that $\frako_K$ embeds in $\calO_\ell(\calI_1)$, so $h_1 = 1$ and $h_2 = 0$.
Again we have $\frac{|P_K(\phi)|^2}{(\phi, \phi)} = \frac 45$.  Then \eqref{eq:lvalQ} 
or \cref{main-thm:Q} implies that
\[ C(\Q(\sqrt{-11}), 11)L_\fin(\frac 12, f \otimes \eta_{\Q(\sqrt{-11})}) = C(\Q(i), 11) L_\fin(\frac 12, f \otimes \eta_{\Q(i)}). \]
Hence
\[ \frac{L_\fin(\frac 12, f \otimes \eta_{\Q(\sqrt{-11})})}{L_\fin(\frac 12, f \otimes \eta_{\Q(i)})}
= \frac{4}{\sqrt{11}},\]
which agrees with numerical calculations.  This provides a check on our local factors
when $K$ is ramified at a prime $p | N_1$.
\end{ex}

\begin{ex} \label{ex:22}
Let $D_B = 11$ and $N = 22$.  Then $n=3$ and we may order the ideal
classes so that $w_1 = 2$,  $w_2 = w_3 = 1$.  We can take an orthogonal basis
of $S(\calO)$ to be $\{ \phi_1, \phi_2 \}$ where $\phi_1(x_1) = \phi_1(x_2) = 2$,
$\phi_1(x_3) = -3$ and $\phi_2(x_1) = 2$, $\phi_2(x_2) = -1$, $\phi_2(x_3) = 0$.
(There is a maximal order $\calO_B \supset \calO$ such that as
double cosets $\Cl(\calO_B) = \{ x_1 \sqcup x_2, x_3 \}$ and $\phi_1$ the
old eigenform in $S(\calO_B)$ denoted $\phi$ in the previous example.)
Here $S(\calO)$ is the $\hat \calO^\times$-invariant subspace of the cuspidal
representation $\pi$ of $B^\times(\A_\Q)$ corresponding to the newform $f \in S_2(11)$.

Suppose $K=\Q(\sqrt{-15})$ so $h_K = 2$.  We compute $h_1 = 0$ and $h_2 = h_3 = 1$.
Hence 
\[ \frac{|P_K(\phi_1)|^2}{(\phi_1,\phi_1)} + \frac{|P_K(\phi_2)|^2}{(\phi_2,\phi_2)} =
\frac 1{15} + \frac 13 = \frac 25.
\]
Since $a_2(f) = -2$, we see $\Lambda_2(f, 1_K) =4$ from \eqref{lambdaMQ}.  
Then \eqref{eq:lvalQ} tells us that
\[ \frac{L_\fin(\frac 12, f)L_\fin(\frac 12, f \otimes \eta_K)}{(f,f)} 
= \frac{8\pi^2}{5 \sqrt{15}}.
\]

If $\chi$ is the nontrivial character of $\Cl(K)$, we see
\[ \frac{|P_{K,\chi}(\phi_1)|^2}{(\phi_1,\phi_1)} + \frac{|P_{K,\chi}(\phi_2)|^2}{(\phi_2,\phi_2)} = \frac {25}{15} + \frac 13 = 2.
\]
Since the Hilbert class field $H_K$ of $K$ is unramified at a prime
$\frakp_2$ of $K$ lying above 2, we see $\chi(\frakp_2) = -1$, and get
$\Lambda_2(f,\chi) = \frac 45$.
Thus \eqref{eq:lvalQ} tells us that
\[ \frac{L_\fin(\frac 12, f, \chi)}{(f,f)} = \frac {8\sqrt{15} \pi^2}3. \]

We numerically verified both of these $L$-values in Magma in
terms of elliptic curve base change  $L$-values.  
Namely, if $E$ is an elliptic curve of conductor 11 associated to $f$, 
then $L(\frac 12, f, 1_K) = L(\frac 12, E_K)$, and 
$L_\fin(\frac 12, f, \chi) = L(\frac 12, E_{H_K}) / L(\frac 12, E_K)$.

While we could have obtained these $L$-values using the maximal order $\calO_B$,
this example provides a check on our formulas \eqref{lambdaM}, \eqref{lambdaMQ}
for the local factors $\Lambda_p(f,\chi)$ when $p | M$. 
\end{ex}

\begin{ex} \label{ex:27}
Let $N=27$, so $D_B = 3$.  Since $\dim S_2(27) = 1$, we have $n = 2$.  
We also have $m(\calO) = \frac 32$,
so we can order $\Cl(\calO)$ such that $w_1 = 2, w_2 = 1$.  Thus $S(\calO)$ is
spanned by the function $\phi$ given by $\phi(x_1) = 2$, $\phi(x_2) = -1$, and
$\phi$ corresponds to the unique newform $f \in S_2(27)$.

Suppose $K = \Q(i)$.  Then $h_1 = 1$ and $h_2 = 0$, so $\frac{|P_K(\phi)|^2}{(\phi,\phi)} =
\frac 43$.  Thus by \cref{main-thm} (or \eqref{eq:lvalQ}) we have
\[ \frac{L_\fin(\frac 12, f)L_\fin(\frac 12, f \otimes \eta_K)}{(f,f)} = \frac {4 \pi^2}3, \]
which matches numerical computations.  Note $N=27$ is excluded from
\cref{thm:prime} and this example shows \eqref{prime:double} does not hold
when $N=27$.
\end{ex}

\begin{ex} Let $N=75$ with $D_B = 5$.  
This corresponds to the triple $(1,25,3)$, which is balanced so each $w_i = 1$.  
Here $n=8$ with $\dim \Eis(\calO) = 2$.  Let $g$ be the unique newform in 
$S_2(15)$.  Note $S_2^\new(75)$ is generated by 3 newforms, say $f_1, f_2, f_3$.
One of them, say $f_1$, is a quadratic twist of $g$, and thus locally a ramified
twist of Steinberg at 5.  The other two, $f_2$ and $f_3$, are quadratic twists of each other
and minimal at 5, whence locally supercuspidal at 5.  We remark that \cref{JL} 
tells us that  $S(\calO) \simeq \C g \oplus \C f_1 \oplus 2 \C f_2 \oplus 2 \C f_3$.

Let $K=\Q(\sqrt{-5})$, so $h_K = 2$ and $K$ is split at 3 and ramified at 5.
Then \cref{main-thm:Q} says
\[ \frac{5 \sqrt 5}{12 \pi^2} \sum_{\chi \in \widehat{\Cl}(K)} \left( \sum_{i=1}^3 \Lambda_{5}(f_i)\frac{L_\fin(\frac 12, f_i, \chi)}{(f_i,f_i)}  +  \frac 15  \frac{L_\fin(\frac 12, g, \chi)}{(g,g)} \right) = 3. \]
From \eqref{lambdaN2}, we see that $\Lambda_{5}(f_1) = \frac{24}{25}$ and
$\Lambda_{5}(f_2) = \Lambda_{5}(f_3) = \frac 65$.
This agrees with numerical calculations of base change elliptic curve $L$-values.
\end{ex}

\section{Appendix: Guide to Hypotheses} \label{sec:appendix}

For the convenience of the reader, 
we present in \cref{table:1} a summary of our primary 
assumptions on $F$, the level type $(\frakN_1$, $\frakN_2$, $\frakM)$ or the level $\calO$,
and the CM extension $K/F$ 
for various global results in this paper.  
Some ``exceptional'' hypotheses 
(e.g., $N \ge 11$ and $N \ne 27$ in \cref{thm:prime}, or coprimality of $D_B$ and 
$D_K$ in \cref{cor:stable}) are excluded for brevity.

We refer to \cref{sec:notation} for general remarks on notation, with
\cref{sec:spec} explaining our standing conventions on 
($\frakN_1$, $\frakN_2$, $\frakM$).  The definitions of
nice level types and admissibility of $K$ are
given in \cref{sec:12} when $F=\Q$, and in \cref{sec:emb} in general.  
The notation (P) in the bottom entries of the table refers to a parity condition
on the number of primes dividing $N_0$ or $N_1$.

\begin{center}
\begin{table}
\begin{tabular}{c|c|c|c}
& $F$ & ($\frakN_1$, $\frakN_2$, $\frakM$) or $\calO$  & CM $K/F$ \\
\hline
\cref{thm:prime} & $\Q$ & ($p^{2r+1}$, 1, 1) & non-split at $p$; \\
&&& inert at $p$ if $r > 0$ \\ 
\cref{cor:12} & $\Q$ & ($p^{2r+1}$, 1, 1) & inert at $p$ \\ 
\cref{main-thm:Q} & $\Q$ & nice level type & admissible \\
\cref{thm2}(1)& $\Q$ & nice level type, balanced & admissible \\
\cref{thm2}(2)& $\Q$ & nice level type & admissible, stable range \\
\cref{thm2}(3)& $\Q$ & nice level type & admissible, $u_K > 1$ \\
\cref{cor:stable} & $\Q$ & nice level type with $N_2 = 1$ & admissible, stable range  \\
\cref{JL}  & --- & $\frakN_2$ cubefree & --- \\
\cref{col-orthog} & --- & --- & --- \\
\cref{prop:per-avg} & --- & --- & $\frako_K \subset \calO$ \\
\cref{cor:per-avg} & --- & --- & $\frako_K \subset \calO$ \\
\cref{cor:per-avg-bal}(i) & --- & $\calO$ balanced  & $\frako_K \subset \calO$ \\
\cref{cor:per-avg-bal}(ii) & $h_F^+=1$ & $\calO$ unram.\ quad.\ & $\frako_K \subset \calO$ \\
\cref{cor:semistable} & --- & --- & $\frako_K \subset \calO$ \\
\cref{pichi-adm} & $h_F = 1$ & nice level type & $\frako_K \subset \calO$ \\
\cref{lval:rel} & $h_F = 1$ & nice level type & admissible \\
\cref{main-thm} & $h_F = 1$ & nice level type & admissible \\
\eqref{eq:lower-bound1} & $\Q$ & $N=N_1$ sqfree, (P) & $(N, 1, 1)$-admissible \\
\cref{lowbdN}(i) &  $\Q$ & $N = p^{2r+1} N_0$, $pN_0$ sqfree, (P)
& $(N, 1, 1)$-admissible \\
\cref{lowbdN}(ii) &  $\Q$ & $N = p^{2} N_0$, $pN_0$ sqfree, (P)
& $(N_0, p^2, 1)$-admissible \\
\cref{lowbdM} & $\Q$ & $N = p N_1$ sqfree, (P)
& $(N_1, 1, p)$-admissible \\
\end{tabular}
\caption{Table of hypotheses.}
\label{table:1}
\end{table}
\end{center}

%
%

\end{document}